\documentclass[12pt]{amsart}
\usepackage{lipsum}
\usepackage[usenames]{color}
\input xypic

\usepackage{amsxtra}
\usepackage{amscd}
\usepackage{amsthm}
\usepackage{amsfonts}
\usepackage{amssymb}
\usepackage{pstricks,pst-node}

\usepackage{comment}

\usepackage{verbatim}
\usepackage{url}
\usepackage{stmaryrd}

\usepackage{amsmath}
\usepackage[centertags]{amsmath}
\usepackage{amsfonts}
\usepackage{amssymb}
\usepackage{amsthm}
\usepackage{newlfont}
\usepackage{url}

\theoremstyle{plain}
\newtheorem{thm}{Theorem}
\newtheorem{cor}[thm]{Corollary}

\newtheorem{lem}[thm]{Lemma}
\newtheorem{lem*}[thm]{Lemma}
\newtheorem{prop}[thm]{Proposition}

\theoremstyle{definition}
\newtheorem{dfn}{Definition}
\theoremstyle{remark}
\newtheorem{rem}{Remark}
\newtheorem{rem*}{Remark}

\newtheorem{ex}[rem]{Example}

\numberwithin{rem}{section} 
\numberwithin{dfn}{section} 
\numberwithin{equation}{section} 
\numberwithin{thm}{section} 

\def\n{\noindent}
\def\ov{\overline}
\def\bs{\backslash}
\def\smatrix{\smallmatrix}

\def\pmatrix{\left(\smatrix}
\def\endpmatrix{\endsmallmatrix\right)}
\def\upi{\pmb{\pi}}

\def\!{\operatorname{!}}
\def\wh{\widehat}

\def\into{\hookrightarrow}
\def\onto{\twoheadrightarrow}

\def\C{\mathbb C}
\def\F{\mathbb F}

\def\H{\mathbb H}

\def\P{\mathbb P}
\def\Q{\mathbb Q}

\def\Z{\mathbb Z}
\def\br{\llbracket}
\def\ebr{\rrbracket}

\def\fb{\mathfrak{b}}
\def\fg{\mathfrak{g}}
\def\fh{\mathfrak{h}}

\def\fn{\mathfrak{n}}
\def\fp{\mathfrak{p}}

\def\AA{\mathcal A}

\def\EE{\mathcal E}
\def\FF{\mathcal F}

\def\HH{\mathcal H}

\def\LL{\mathcal L}

\def\NN{\mathcal N}
\def\OO{\mathcal O}

\def\1{\mathbf 1}

\def\Ad{\operatorname{Ad}}
\def\ad{\operatorname{ad}}

\def\diag{\operatorname{diag}}

\def\Hom{\operatorname{Hom}}
\def\GL{\operatorname{GL}}

\def\tr{\operatorname{tr}}
\def\str{\operatorname{str}}
\def\sl{\operatorname{sl}}
\def\Sp{\operatorname{Sp}}

\def\id{\operatorname{id}}
\def\Id{\operatorname{Id}}

\def\End{\operatorname{End}}

\def\wt{\widetilde}

\def\varep{\varepsilon}

\def\sgn{\operatorname{sgn}}

\def\gl{\operatorname{gl}}
\def\ev{\operatorname{ev}}
\def\ind{\operatorname{ind}}
\def\Par{\operatorname{Par}}
\def\0{\ov{0}}
\def\1{\ov{1}}
\def\tin{{n'}}
\def\even{\operatorname{even}}
\def\odd{\operatorname{odd}}
\def\dist{\operatorname{dist}}
\def\QS{\operatorname{QS}}
\def\Aa{\operatorname{AA}}
\def\AI{\operatorname{AI}}

\def\varep{\operatorname{\varepsilon}}

\def\ex{\operatorname{ex}}
\def\ove{\ov{\varep}}

\def\ovd{\ov{d}}

\begin{document}

\title{Affine quantum super Schur-Weyl duality}

\author{Yuval Z. Flicker}
\begin{abstract}
The Schur-Weyl duality, which started as the study of the commuting actions of the symmetric group $S_d$ and $\GL(n,\C)$ on $V^{\otimes d}$ where $V=\C^n$, was extended by Drinfeld and Jimbo to the context of the finite Iwahori-Hecke algebra $H_d(q^2)$ and quantum algebras $U_q(\gl(n))$, on using universal $R$-matrices, which solve the Yang-Baxter equation. There were two extensions of this duality in the Hecke-quantum case: to the {\em affine} case, by Chari and Pressley, and to the {\em super} case, by Moon and by Mitsuhashi. We complete this chain of works by completing the cube, dealing with the general {\em affine super} case, relating the commuting actions of the affine Iwahori-Hecke algebra $H^a_d(q^2)$ and of the affine quantum Lie superalgebra $U_{q,a}^\sigma(\sl(m,n))$ using the presentation by Yamane in terms of generators and relations, acting on the $d$th tensor power of the superspace $V=\C^{m+n}$. Thus we construct a functor and show it is an equivalence of categories of $H_d^a(q^2)$ and $U_{q,a}^\sigma(\sl(m,n))$-modules when $d<n'=m+n$.
\end{abstract}
\date{\today}

\address{Ariel University, Ariel 40700, Israel; The Ohio State University, Columbus OH43210.}
\email{yzflicker@gmail.com}
\thanks{This is a slightly improved exposition of the article which appeared Dec. 1, 2018 online in "Algebras and Representation Theory", http://doi.org/10.1007/s10468-018-9841-1.\\ 
\indent I wish to express my deep gratitude to the referee and to P. Deligne for carefully reading this work.\\ 
\indent Partially supported by the Simons Foundation grant \#317731. This work was partially carried out at MPIM, Bonn; YMSC, QingHua University, Beijing; and the Hebrew University, Jerusalem.}

\maketitle

\tableofcontents
\section{Introduction} \label{1}
The finite dimensional irreducible representations of the general linear group $\GL(n,\C)$, or equivalently its Lie algebra $\gl(n,\C)$, where $n$ is a positive integer and $\C$ is an algebraically closed field of characteristic zero, can be classified as highest weight modules, constructed as quotients of Verma modules. This applies to any semisimple Lie algebra, and was extended to classical semisimple Lie superalgebras by Kac \cite{k77}. 

Another approach was introduced by Schur \cite{sch01}, and differently in \cite{sch27}, who considered the permutation action of the symmetric group $S_d$ on $d\ge 1$ letters, and the diagonal action of $\GL(n,\C)\simeq\GL(V)$, on $V^{\otimes d}$. Schur proved that these two actions have a double centralizing property in $\End(V^{\otimes d})$. Representations of $\GL(V)$ are thus determined from those of $S_d$.

Denote Schur's representation by $\pi_d:S_d\to\End(V^{\otimes d})$. The group algebra $\C S_d$ decomposes as $\oplus_{\lambda\in\Par(d)}I_\lambda$, where $I_\lambda$ are simple algebras, and $\Par(d)$ is the set of partitions $\lambda=(\lambda_1,\lambda_2,\dots)$, $\lambda_i\ge\lambda_{i+1}\ge 0$, of $d$: $\sum\lambda_i=d$. Hence there is a subset $\Gamma(n;d)\subset\Par(d)$ such that $\pi_d(\C S_d)\simeq\oplus_{\lambda\in\Gamma(n;d)}I_\lambda$, so that the $\GL(V)$-irreducible representations which appear in $\End(V^{\otimes d})$ are precisely those associated to $\Gamma(n;d)$. This gives a bijection between the set of representations of $\GL(n)$ which appear in $V^{\otimes d}$ and a { subset} of the set of irreducible representations of $S_d$, known by the work of Frobenius \cite{f00} and Young \cite{yg}. Schur's work was continued by Weyl \cite{w53}, who determined the $\Gamma(n;d)$ and proved Weyl's strip theorem, which asserts that $\lambda=(\lambda_1,\lambda_2,\dots)\in\Gamma(n;d)$ iff $\lambda_j=0$ for $j>n$. In particular, if $n\ge d=\lambda_1+\lambda_2+\dots$ then $\lambda_j=0$ for all $j>n$, hence  $\Gamma(n;d)=\Par(d)$, so there is a canonical bijection between the set of irreducible representations of $\GL(n,\C)$ in $V^{\otimes d}$, and the set of irreducible representations of $S_d$. 

Multiplying $\pi_d: S_d\to\End(V^{\otimes d})$ with $\sigma\mapsto\sgn(\sigma)$, $S_d\onto\{\pm1\}$, one gets the same results, but with the vertical strips $\Gamma(m;d)'$ (where $\lambda'$ is the transpose of $\lambda$) replacing the horizontal strips $\Gamma(m;d)$. Gluing these two permutations actions of $S_d$ on $V^{\otimes d}$, Berele and Regev \cite{br87} studied the permutation action of $S_d$ on $V^{\otimes d}$, where $V=V_{\0}\oplus V_{\1}$, $\dim V_{\0}=m$, $\dim V_{\1}=n$, such that the restricted action on $V_{\0}$ permutes without a sign, and on $V_{\1}$ with $\sgn(\sigma)$. Obtained are the Young diagrams not containing the box $(m+1,n+1)$. This yields a representation $\wh\pi_d:S_d\to\End(V^{\otimes d})$, and a subset $\Gamma(m,n;d)\subset\Par(d)$ with $\wh\pi_d(\C S_d)\simeq\oplus_{\lambda\in\Gamma(m,n;d)}I_\lambda$. The Hook theorem \cite[Theorem 3.20]{br87} asserts that $\lambda=(\lambda_1,\lambda_2,\dots)\in\Gamma(m,n;d)$ iff $\lambda_j\le n$ for $j>m$, i.e., $\lambda_{m+1}<n+1$. Note that if $d<(m+1)(n+1)$ then $d=|\lambda|=\sum_{i\ge 1}\lambda_i\ge\sum_{1\le i\le m+1}\lambda_i\ge (m+1)\lambda_{m+1}$ implies $\lambda_{m+1}<n+1$.

Another way to state Schur's work is as follows. Let $(\rho,V)$ be the natural $n$-dimensional representation of $\GL(n,\C)$, and $\rho_d$ the diagonal representation on $V^{\otimes d}$. This action of $\GL(n,\C)$ commutes with the permutation action $\pi_d$ of $S_d$ on $V^{\otimes d}$. Thus to any right $S_d$-module $M$ there is a $\GL(n,\C)$-module $F(M)=M\otimes_{S_d}V^{\otimes d}$. The Schur-Weyl theory asserts that for $d\le n$, the functor $M\mapsto F(M)$ defines an equivalence from the category of $S_d$-modules of finite rank, to the category of finite rank $\GL(n,\C)$-modules whose irreducible constituents all occur in $V^{\otimes d}$.

Drinfeld and Jimbo introduced, independently, in 1985, a family of Hopf algebras $U_q(\fg)$, depending on a parameter $q\in \C^\times$, associated to any symmetrizable Kac-Moody algebra $\fg$. For $q$ not a root of unity, Jimbo \cite{j86} announced an analogue of the Schur-Weyl duality with the quantum group $U_q(\gl(n))$ replacing $\gl(n,\C)$, $V$ replaced by the natural $n$-dimensional irreducible representation of $U_q(\gl(n))$, and $S_d$ by its Hecke algebra $H_d(q^2)$. The latter is isomorphic to the group algebra $\C(q)S_d$ of $S_d$ over the field $\C(q)$ (see Proposition \ref{P7.2} for a precise statement). The representation of this Hecke algebra on $V^{\otimes d}$ is defined by the $R$-operators, or ``universal $R$-matrix", which is the solution for the quantum Yang-Baxter equation, and satisfies the relations of the generators defining the Hecke algebra.

The Hecke algebra, $H_d(q^2)$, also called the finite Iwahori-Hecke algebra, is the finite part of the general affine Iwahori-Hecke algebra, $H^a_d(q^2)$, which for prime-power $q$ is the convolution algebra $C_c[I\bs G/I]$ of the compactly supported $\C$-valued functions on the group $G$ of points over a local non-Archimedean field $F$ whose residual cardinality is $q$, of a reductive connected $F$-group, which are bi-invariant under the action of an Iwahori subgroup $I$ of $G$. The finite Iwahori-Hecke algebra $H_d(q^2)$ is just the subalgebra corresponding to $C_c[I\bs K/I]$, where $K$ is a maximal compact subgroup of $G$. The affine algebra is of great importance (when $q$ is a prime-power) for automorphic forms and neighboring areas. It was given a presentation in terms of generators and relations by Iwahori and Matsumoto \cite{im65}, and another one -- which reflects better the structure of the double coset space $I\bs G/I$, by J. Bernstein \cite{hkp10}. These presentations make sense for all $q$.

Drinfeld suggested in \cite{d86} that the Schur-Weyl theory should extend to relate the affine Hecke algebra $H^a_d(q^2)$ and the affine quantum algebra $U_q(\wh\sl(n))$. This was done by Chari-Pressley \cite{cp96}, who constructed a functor from the category of finite-rank $H^a_d(q^2)$-modules to the category of completely decomposable finite rank $U_q(\wh\sl(n))$-modules whose irreducible constituents occur in $V^{\otimes d}$, when $q$ is not a root of unity, extending Jimbo's functor \cite{j86} relating the non-affine $U_q(\gl(n))$ and $H_d(q^2)$; see also \cite{g86}. A suitable limit as $q\to 1$ gives Drinfeld's \cite{d86} (see also \cite{d88}) Schur-Weyl duality for the Yangian $Y(\gl(n))$, where the role of $S_d$ is played by a degenerate affine Hecke algebra whose defining relations are obtained from those of $H^a_d(q^2)$ for some $q\to 1$.

A ``super" extension of Jimbo's work \cite{j86} to the context of the quantum superalgebra $U_q(\gl(m,n))$, where the Hecke algebra $H_d(q^2)$ remains, but its action is composed with a sign character, or alternatively a quantum extension of the work of Berele-Regev \cite{br87}, thus the action of $S_d$ is replaced by that of the finite Iwahori-Hecke algebra, and that of $\GL(n,\C)$ by that of the quantum superalgebra $U_q(\gl(m,n))$, was done by Moon \cite{mo03}, and also by Mitsuhashi \cite{mi06}. Both  use the result of Benkart, Kang, Kashiwara \cite{bkk00} which shows the complete reducibility of the tensor product $V^{\otimes d}$ of the natural representation $V$ of $U_q(\gl(m,n))$ using the crystal base theory of $U_q(\gl(m,n))$.

However, it is the action of the {\em affine} Hecke algebra which is the most interesting. Our aim here is to complete this missing general case by extending the Schur-Weyl duality to relate the action of the affine Iwahori-Hecke algebra $H^a_d(q^2)$ with that of the affine quantum Lie superalgebra $U_{q,a}^\sigma(\sl(m,n))$, thus extending the functor constructed by Jimbo and Chari-Pressley to the context of the affine Hecke algebra and the affine quantum Lie {\em super}algebra $U_{q,a}^\sigma(\sl(m,n))$, or alternatively the work of Moon and Mitsuhashi to the {\em affine} quantum Lie superalgebra case. This is the natural, general case. 

A necessary ingredient is a definition of the quantum affine Lie superalgebra in terms of generators and relations. This is provided by the work of Yamane \cite{y99}. In this affine super case there are new relations: $(\QS4)(4)$ and $(\QS5)(4)$, that do not appear in the non-super case, and we need to verify that they too are satisfied by the operators that we introduce. This is a novelty of the affine super case.

As in \cite[Theorem 4.2]{cp96}, to extend \cite{mo03} and \cite{mi06} to the affine case one needs to verify the relations which are new to the affine case, satisfied by the additional generators, $x_0^{\pm}$, or $E_0$ and $F_0$. Naturally our results can be used to obtain equivalence of categories of representations of affine Iwahori-Hecke algebras and affine quantum Lie superalgebras, as done in \cite{cp96} in the non-graded case. We prefer to leave this for a sequel, as well as other applications we have in mind.

In \cite{k14} (see also \cite{kkk13}) the main philosophy and results of categorification and 2-representation theory, and the quantum affine Schur-Weyl duality in this language is explained. The Khovanov-Lauda-Rouquier algebras play a central role. This is an interesting direction of further work. For recent survey of related work, and directions of current research, from relations to geometry by Maulik-Okounkov, to categorification of cluster algebras using $R$-matrices by Kang-Kashiwara-Kim-Oh, see \cite{h17}. For representation theory of $U_{q,a}^\sigma(\sl(m,n))$ see \cite{zh17} and references there, as well as \cite{zr93}. We hope there is still some interest in our modest but explicit construction.

Perhaps the most interesting fact about the Schur-Weyl duality in this quantun-Hecke setting is the unexpected connection between the affine Iwahori-Hecke algebra, which comes from number theory and automorphic forms for prime-power $q$, on one hand, and the quantum theory of Yang-Baxter equations, which affords the action of the Hecke algebra via the $R$-operators, for general $q$, originating from physics, on the other hand. The parameter $q$ is the residual cardinality of the local field from the arithmetic perspective, and can be interpreted as the temperature from the physical point of view.

My initial motivation to study this area was to understand Drinfeld's ideas on the Yang-Baxter equation and on quantum groups. I was fascinated by the words -- not knowing their meaning -- since I studied his ``elliptic modules". Clearly there is a strong resemblance between the Schur-Weyl duality, and the Galois-Automorphic duality that Drinfeld studied in ``elliptic modules". Another push came from a very brief conversation with Eric Opdam who mentioned to me his work (\cite{ho97}). This led me to realize that the area concerns Hecke algebras, with which I am familiar. The final nail came from a brief social conversation with Mikhail Kapranov that led me later to read his \cite{k18}, and then to Manin \cite{m97} and to Deligne-Morgan \cite{dm99} notes on Bernstein's lectures at IAS, which made me realize the significance of supersymmetry.

\section{Superalgebras}\label{2}
Let $m$, $n\ge 1$ be positive integers. Put $\tin=m+n$, $n'{}'=n'-1$. Let $R$ be a field of characteristic zero. For a fuller exposition to superalgebras see \cite{dm99}.

The {\em general linear} (Lie) {\em superalgebra} $\fg=\gl(m,n)$ over $R$ is the algebra $M(\tin\times\tin,R)$ of $\tin\times\tin$ matrices over $R$, $\Z/2$-graded as $\gl(m,n)_{\0}\oplus\gl(m,n)_{\1}$, where 
\[
\gl(m,n)_{\0}=\{\diag(A,D);\,A\in M(m\times m,R),\,D\in M(n\times n,R)\},
\] 
and
\[
\gl(m,n)_{\1}=\left\{\pmatrix 0&B\\ C&0\endpmatrix;\, B\in M(m\times n,R),\, C\in M(n\times m,R)\right\},
\]
with the bilinear {\em super bracket} $[x,y]=xy-(-1)^{ab}yx$ for $x\in\gl(m,n)_{\ov{a}}$, $y\in\gl(m,n)_{\ov{b}}$, $a$, $b\in\{0,\,1\}$, on $\gl(m,n)$. An element of $\gl(m,n)$ is called {\em homogeneous} if it lies in $\gl(m,n)_{\ov{a}}$.

Define a {\em parity function} $p$ by $p(x)=a$ if $0\not=x\in\gl(m,n)_{\ov{a}}$, $a\in\{0,\,1\}$.

Define the {\em supertrace} $\str\pmatrix A&B\\ C&D\endpmatrix=\tr A-\tr D$ on $\gl(m,n)$, where $\tr$ is the usual trace. 

Put $\sl(m,n)=\ker \str$. Put $I=\{1,2,\dots,\tin\}$, $I'=\{1,2,\dots,n'{}'=\tin-1\}$.

Let $E_{i,j}\in\gl(m,n)$ be the matrix whose only nonzero entry is 1 at the $(i,j)$-position.

The {\em Cartan subalgebra} $\fh$ of $\gl(m,n)$ is the $R$-span $\Sp_R\{E_{i,i};\,i\in I\}$, namely the algebra of diagonal matrices.

Let $h_i \in\fh$ $(i\in I')$ be $E_{i,i}-(-1)^{p(i)}E_{i+1,i+1}$, where $p(m)=1$ and $p(i)=0$ for $i\not=m$.

Denote by $\fh^\ast=\Hom(\fh,R)$ the {\em dual space of} $\fh$. Under the adjoint action ($\Ad(h)y=[h,y]$) of $\fh$, $\gl(m,n)$ decomposes as a direct sum of {\em root spaces} $\fh\oplus\oplus_{\alpha\in\fh^\ast}\gl(m,n)_{\alpha}$, where 
\[
\gl(m,n)_{\alpha}=\{X;\,\Ad(h)X=\alpha(h)X,\,\forall h\in\fh\}.
\]
An $\alpha\in\fh^\ast-\{0\}$ is called a {\em root} if the root space $\gl(m,n)_{\alpha}$ is not zero.

Let $\{\varep_i;\,i\in I\}$ denote the standard basis of $R^{\tin}=R^m\oplus R^n$. In particular $\{\varep_i;\,1\le i\le m\}$ is the standard basis of $R^m$, and $\{\varep_i;\,m< i\le \tin\}$ of $R^n$. The {\em simple roots} are $\alpha_i=\varep_i-\varep_{i+1}\in\fh^\ast$, $i\in I'$, and the {\em fundamental weights} are $\varpi_i=\varep_1+\dots+\varep_i\in\fh^\ast$, $i\in I'$. A root $\alpha$ is {\em even} if $\gl(m,n)_\alpha\cap\gl(m,n)_{\0}\not=0$, and {\em odd} if $\gl(m,n)_\alpha\cap\gl(m,n)_{\1}\not=0$. Thus all simple roots are even, except $\alpha_m$, which is odd.

The {\em lattice of integral weights} $P\subset\fh^\ast$ is the $\Z$-span of $\{\varep_i;\,i\in I\}$. The {\em dual weight lattice} $P^\vee\subset\fh$ is the free $\Z$-module spanned by $E_{i,i}$, $i\in I$. For $\lambda\in\fh^\ast$, $h\in\fh$ define $\lambda(h)=\langle h,\lambda\rangle$ by linearity and $\varep_i(E_{j,j})=\delta_{i,j}=\delta(i,j)$ (= 1 if $i=j$; =0 if $i\not=j$). We get a natural pairing $\langle .,. \rangle:\fh\times\fh^\ast\to R$ with $\langle h_i,\alpha_j\rangle=\alpha_j(h_i)$. The {\em Cartan matrix} $A=(a_{ij}=\alpha_j(h_i);\,i,j\in I')$ has nonzero entries: 2 at each diagonal place $(i,i)\not=(m,m)$; $-1$ at each underdiagonal place $(i+1,i)$ and at each over diagonal place $(i,i+1)$, except at $(m,m+1)$ where the entry is 1. The entries at $(m,m)$ and at $(i,j)$ with $|i-j|\ge 2$ are zero. The Cartan matrix is {\em symmetrizable} in the sense that $DA$ is symmetric, where $D=\diag(I_m,-I_{n-1})$, and $I_k$ denotes the identity $k\times k$ matrix. Denote the diagonal entries of $D$ by $d_i$, thus $D=\diag(d_1,\dots,d_{\tin-1})$, and $d_i=1$ ($1\le i\le m$), $d_i=-1$ ($m< i<\tin$). Explicitly
\[
A=\pmatrix
2&-1&&&&&&&0\\
-1&2&-1&&&&&&0\\
0&&&&&&&&0\\
0&&-1&0&1&&&&0\\
0&&&-1&2&-1&&&0\\
0&&&&&&&&0\\
0&&&&&&-1&2&-1\\
0&&&&&&&-1&2
\endpmatrix,
\qquad
DA=\pmatrix
2&-1&&&&&&&0\\
-1&2&-1&&&&&&0\\
0&&&&&&&&0\\
0&&-1&0&1&&&&0\\
0&&&1&-2&1&&&0\\
0&&&&&&&&0\\
0&&&&&&1&-2&1\\
0&&&&&&&1&-2
\endpmatrix
\]
The middle rows are the $m$th and $(m+1)$st, $DA$ is symmetric. We also put $d_0=d_{n'}=-1$.

The above is the example with which we deal. In general $\fg$ will be the contragredient Lie superalgebra corresponding as in \cite{k77}, \cite{k78} to the following data $(I',\,P,\,\{\alpha_i\},\,\{H_i\},\,A,\,D$, $\langle x,y\rangle,\,(x,y))$. Let $I'$ be the index set for the simple roots. Assume it is partitioned into two parts corresponding to the even simple roots and the odd simple roots:
\[
I'=I'_{\even}\cup I'_{\odd}.
\]
Write $p(i)=0$ if $i\in I'_{\even}$, and $p(i)=1$ if $i\in I'_{\odd}$. In our example of $\fg=\gl(m,n)$, $I'_{\odd}=\{m\}$, $I'_{\even}=I'-\{m\}$, $I'=\{1,\dots,m+n-1=n'{}'\}$.

Let $P$ be a free $\Z$-module -- its elements are called the {\em integral weights} -- given with a $\Q$-valued symmetric bilinear form $(.,.)$. Also given for each $i\in I'$ is the {\em simple root} $\alpha_i\in P$ and the {\em simple coroot} $h_i\in P^\vee$. In the natural pairing $\langle .,. \rangle:\,P^\vee\times P\to\Z$ they are assumed to satisfy: $\langle h_i,\alpha_i\rangle=2$ if $i\in I'_{\even}$; $\langle h_i,\alpha_i\rangle=0$ or 2 if $i\in I'_{\odd}$; $\langle h_i,\alpha_j\rangle\le 0$ if $j\not=i$ and $p(i)=p(j)$. Also given are nonzero integers $d_i$ with
\[
d_i\langle h_i,\lambda\rangle=(\alpha_i,\lambda)\qquad\forall\lambda\in P.
\]
Since $d_i\langle h_i,\alpha_j\rangle=(\alpha_i,\alpha_j)=(\alpha_j,\alpha_i)=d_j\langle h_j,\alpha_i\rangle$, the Cartan matrix 
\[
A=(a_{ij}=\langle h_i,\alpha_j\rangle;\,i,j\in I')
\]
is symmetrizable: $DA$ is symmetric if $D=\diag(d_i;\,i\in I')$.

More generally, one may consider a vector space $V$ of dimension $(m,n)$, and a homogeneous basis $(\varep_i;\,1\le i\le m+n)$ with $\varep_i$ of parity $p(i)$. Each such ordered basis gives rise to a deformation of the enveloping algebra of $\gl(V)$ as in the next section. But we consider here only the case where all even $\varep_i$ are placed first, for simplicity.

\section{Quantum superalgebras}\label{3}
Following \cite{bkk00} (and its predecessors \cite{kt91}, \cite{flv91}, \cite{st92}, \cite{st93}, \cite{y94}, \cite{y99}, and \cite{zr14}) we now introduce the $q$-deformation $U_q(\fg)$ of the universal enveloping algebra of the contragredient Lie superalgebra $\fg$ corresponding as in \cite{k77}, \cite{k78} to the data of last section. Denote by $q$ an indeterminate, put $q_i=q^{d_i}$. Define the bilinear form $[x,y]_v$ to be $xy-(-1)^{p(x)p(y)}vyx$ on homogeneous $x$, $y$; note that $[.,.]=[.,.]_1$. The associated quantum enveloping algebra $U'_q(\fg)$ is the associative algebra over $\Q(q)$ with 1, generated by $e_i$, $f_i$ ($i\in I'$), and $q^h$ ($h\in P^\vee$), satisfying
\[
q^h=1 \quad\mathrm{for}\quad h=0; \qquad q^{h+h'}=q^hq^{h'} \quad\mathrm{for}\quad h,\,h'\in P^\vee;
\]
\[
q^he_i=q^{\langle h,\alpha_i\rangle}e_iq^h,\quad q^hf_i=q^{-\langle h,\alpha_i\rangle}f_iq^h\quad\mathrm{for}\quad h\in P^\vee,\,i\in I';
\]
\[
[e_i,f_j]=e_if_j-(-1)^{p(i)p(j)}f_je_i \quad\mathrm{is}\quad \delta(i,j)\frac{q^{h_i}-q^{-h_i}}{q-q^{-1}}\quad\mathrm{for}\quad i,j\in I';
\]
(note that the notation $x_i^+$ and $x_i^-$ is often used for $e_i$ and $f_i$);\\ 
and the bitransitivity conditions \cite[p. 19]{k77} (We first consider the previous relations, and if in the resulting algebra $a$ satisfies the following property, then we put $a=0$):\\
If $a\in\sum_{i\in I'} U'_q(\fn_+)e_i U'_q(\fn_+)$ satisfies $f_ia\in U'_q(\fn_+)f_i$ for all $i\in I'$ then $a=0$.\\
If $a\in\sum_{i\in I'} U'_q(\fn_-)f_i U'_q(\fn_-)$ satisfies $e_ia\in U'_q(\fn_-)e_i$ for all $i\in I'$ then $a=0$.\\
Here $U'_q(\fn_+)$ is the subalgebra of $U'_q(\fg)$ generated by $\{e_i;\,i\in I'\}$, and $U'_q(\fn_-)$ by $\{f_i;\,i\in I'\}$.

Now $U'_q(\fg)$ is a Hopf superalgebra whose comultiplication $\Delta$, counit $\varep$, antipode $S$ are
\[
\Delta(q^h)=q^h\otimes q^h,\quad \Delta(e_i)=e_i\otimes q^{-d_ih_i}+1\otimes e_i,\quad \Delta(f_i)=f_i\otimes 1+q^{d_ih_i}\otimes f_i;
\]
\[
\varep(q^h)=1,\,\,\,\varep(e_i)=0=\varep(f_i);\,\,\, S(q^{\pm h})=q^{\mp h},\,\,\, S(e_i)=-e_iq^{d_ih_i},\,\,\, S(f_i)=-q^{-d_ih_i}f_i.
\]
Here $h\in P^\vee$, $i\in I'$. It is only the Hopf superalgebra structure which matters to us. It is not a Hopf algebra. One can introduce a Hopf algebra structure by replacing $U'_q(\fg)$ by $U_q(\fg)=U'_q(\fg)\rtimes\langle\sigma\rangle$, where $\sigma$ is the involution on $U'_q(\fg)$ given by:
\[
\sigma(q^h)=q^h,\qquad\sigma(e_i)=(-1)^{p(i)}e_i,\qquad\sigma(f_i)=(-1)^{p(i)}f_i,
\]
where again $h\in P^\vee$, $i\in I'$. Then $U_q(\fg)$ is a Hopf algebra with comultiplication $\Delta_\sigma$, counit $\varep_\sigma$, antipode $S_\sigma$ given by
\[
\Delta_\sigma(\sigma)=\sigma\otimes\sigma;\qquad\Delta_\sigma(q^h)=q^h\otimes q^h;
\]
\[
\Delta_\sigma(e_i)=e_i\otimes q^{-d_ih_i}+\sigma^{p(i)}\otimes e_i,\qquad \Delta_\sigma(f_i)=f_i\otimes 1+\sigma^{p(i)}q^{d_ih_i}\otimes f_i;
\]
\[
\varep_\sigma(\sigma)=1=\varep_\sigma(q^h),\qquad\varep_\sigma(e_i)=0=\varep_\sigma(f_i);
\]
\[
S_\sigma(\sigma)=\sigma;\,\,\, S_\sigma(q^{\pm h})=q^{\mp h},\,\,\, S_\sigma(e_i)=-\sigma^{p(i)}e_iq^{d_ih_i},\,\,\, S_\sigma(f_i)=-\sigma^{p(i)}q^{-d_ih_i}f_i.
\]

In our case of $\fg=\gl(m,n)$, the bitransitivity conditions take the form (see \cite[\S 6, Prop. 6.5.1, 6.7.1]{y99}:\\ 
{\em quadratic relations} (missing in \cite[bottom of p. 669]{zh14}):
\[
[e_i,e_j]=0=[f_i,f_j] \quad\mathrm{if}\quad \dist(i,j)\not=1,\quad i,\,j\in I';
\]
here $\dist(i,j)=|i-j|$; when $i=j=m$ this is:
\[
e_m^2=0=f_m^2;
\]
{\em cubic relations}:
\[
[e_i,[e_i,e_j]_{q^{-1}}]_q=e_i^2e_j-(q+q^{-1})e_ie_je_i+e_je_i^2 \quad\mathrm{is}\quad 0 \quad\mathrm{if}\quad \dist(i,j)=1 \quad\mathrm{and}\quad i\not=m;
\]
same as last line with $e_i$, $e_j$ replaced by $f_i$, $f_j$, or $x_i^{\pm}$, $x_i^{\pm}$, $i,\,j\in I'$;\\
{\em quartic relations} (use $e_m^2=0=f_m^2$ for the $=$), $m$, $n>1$:
\[
[[[e_{m-1},e_m]_q,e_{m+1}]_{q^{-1}},e_m]
=e_{m+1}e_me_{m-1}e_m+e_{m-1}e_me_{m+1}e_m
\]
\[
+e_me_{m-1}e_me_{m+1}+e_me_{m+1}e_me_{m-1}
-(q+q^{-1})e_me_{m-1}e_{m+1}e_m \quad\mathrm{is}\quad 0;
\]
same as the last formula with $e$ replaced by $f$, or $x^{\pm}$.

\section{Weight modules}\label{4}
A $U_q(\fg)$-module $M$ is called a {\em weight module} if it admits a weight space decomposition $M=\oplus_{\lambda\in P}M_\lambda$, where $M_\lambda=\{u\in M;\,q^hu=q^{\langle h,\lambda\rangle}u\,\,\,\forall h\in P^\vee\}$.  A weight module is a {\em highest weight module} with {\em highest weight} $\lambda\in P$ if there exists a unique (up to a scalar multiple) superhomogeneous ($\sigma u_\lambda=\pm u_\lambda$) nonzero vector $u_\lambda\in M$, called a {\em highest weight vector}, such that $M=U_q(\fg)u_\lambda$; $e_iu_\lambda=0$ for all $i\in I'$; $q^hu_\lambda=q^{\langle h,\lambda\rangle}u_\lambda$ for all $h\in P^\vee$. Denote the irreducible highest weight module with highest weight $\lambda$ by $V(\lambda)$. In our case of $\fg=\gl(m,n)$, the set of dominant integral weights is
\[
\Lambda=\left\{\lambda=\sum_{1\le i\le\tin}\lambda_i\varep_i;\,\lambda_i\in\Z,\,\lambda_1\ge\lambda_2\ge\dots\ge\lambda_m,\,\lambda_{m+1}\ge\dots\ge\lambda_{m+n}\right\}.
\]

Let $R=\Q(q)$ and $V=V_{\0}\oplus V_{\1}$ be a $\Z/2$-graded vector space, where $V_{\0}=\oplus_{1\le i\le m}Rv_i$, $V_{\1}=\oplus_{m< i\le\tin}Rv_i$, $\tin=m+n$. Write $p(v_i)=p(i)=0$ if $1\le i\le m$, and $=1$ if $m<i\le\tin$. 

The fundamental representation $(\rho,V)$ of $U_q(\gl(m,n))$ is irreducible and has highest weight $\varep_1$. It is defined by
\[
\rho(\sigma)v_i=(-1)^{p(i)}v_i\,\,\,(i,j\in I);\,\,\,\rho(q^h)v_i=q^{\langle h,\varep_i\rangle}v_i\,\,\,(h\in P^\vee);
\]
\[
\rho(e_i)v_j=\delta(j,i+1)v_{j-1};\,\,\,\rho(f_i)v_j=\delta(i,j)v_{j+1};
\]
where $v_0=0=v_{\tin+1}$. In the basis $\{v_i;\,i\in I\}$ of $V$, $\rho(e_i)=E_{i,i+1}$, $\rho(f_i)=E_{i+1,i}$, $\rho(q^h)=\diag(q^{\langle h,\varep_i\rangle})$. This representation extends to a representation $\rho_d$ of $U_q=U_q(\gl(m,n))$ on $V^{\otimes d}$ via the map
\[
\Delta^{(k)}=(\Delta_\sigma\otimes\Id^{\otimes(k-1)})\Delta^{(k-1)}:U_q\to U_q^{\otimes(k+1)},
\]
where $\Delta^{(1)}=\Delta_\sigma:U_q\to U_q^{\otimes 2}$. Thus we put
\[
\rho_d(x)=\rho^{\otimes d}\circ\Delta^{(d-1)}(x),\quad x\in U_q=U_q(\gl(m,n)).
\]
Explicitly $\rho_d:U_q\to\End_{\Q(q)}(V^{\otimes d})$ is given by
\[
\rho_d(\sigma)=\rho(\sigma)^{\otimes d},\qquad \rho_d(q^h)=\rho(q^h)^{\otimes d}\,\,\,(h\in P^\vee),
\]
\[
\rho_d(e_i)=\sum_{1\le k\le d}\rho(\sigma^{p(i)})^{\otimes(k-1)}\otimes\rho(e_i)\otimes\rho(q^{-d_ih_i})^{\otimes(d-k)}\quad(i\in I'),
\]
\[
\rho_d(f_i)=\sum_{1\le k\le d}\rho(\sigma^{p(i)}q^{d_ih_i})^{\otimes(k-1)}\otimes\rho(f_i)\otimes\Id^{\otimes(d-k)}\quad(i\in I').
\]
\begin{prop} $($\cite[Prop 1]{zr98}, \cite[Prop. 3.1]{bkk00}$)$. \label{P4.1} 
$\rho_d$ is a completely reducible representation of $U_q(\gl(m,n))$ on $V^{\otimes d}$, $d\ge 1$.
\end{prop}

\section{Affine Lie superalgebras}
We now proceed to describe the quantum affine Lie superalgebra $U^\sigma_q=U_{q,a}^\sigma(\sl(m,n))$ which is the main object of study in this work. It will be defined using generators and relations following \cite{y99}. To ease the comparison with \cite{y99}, note that our $(n'=m+n,n'{}'=n'-1,m)$ are $(N,n,N-m)$ in \cite{y99}. In this section we describe the non-quantum case.

Let $\Z/2=\{\0,\1\}$ be the cyclic group of order 2. Let $V=V(\0)\oplus V(\1)$ be a $\Z/2$-graded vector space. An $X\in V(\ov{i})$, $i=0,\,1$, is called \emph{homogeneous} of (degree $i$ and) {\em parity} $p(X)=i$. A {\em Lie superalgebra} is a $\Z/2$-graded $\C$-space $\fg=\fg(\0)\oplus\fg(\1)$ with a bilinear map $[\cdot ,\cdot]:\fg\times\fg\to\fg$, called a {\em Lie superbracket}, such that for homogeneous elements $X$, $Y$, $Z$ we have
\[
[X,Y]=-(-1)^{p(X)p(Y)}[Y,X]\qquad\in\quad\fg(\ov{p(X)+p(Y)}),\]
\[ [X,[Y,Z]]=[[X,Y],Z]+(-1)^{p(X)p(Y)}[Y,[X,Z]].
\]
An {\em invariant form} on $\fg$ is a bilinear form $(\cdot |\cdot):\fg\times\fg\to\C$ satisfying, for homogeneous $X$, $Y$, $Z$  in $\fg$, supersymmetry and Lie invariance, namely
\[
(X|Y)=(-1)^{p(X)p(Y)}(Y|X),\qquad ([X,Y]|Z)=(X|[Y,Z]).
\]
For $X\in\fg$, define $\ad(X):\fg\to\fg$ by $(\ad(X))(Y)=[X,Y]$. Put $\LL\fg=\fg\otimes_\C\C[t,t^{-1}]$.

Following \cite{vdl89}, inspired by \cite{k90} in the non-super case, define a Lie superalgebra $\wh\fg=\wh\fg(\0)\oplus\wh\fg(\1)$ by $\wh\fg(\0)=\LL\fg(\0)\oplus\C c\oplus\C d$, $\wh\fg(\1)=\LL\fg(\1)$ and
\[
[X\otimes t^r+a_1c+b_1d,Y\otimes t^s+a_2c+b_2d]=[X,Y]\otimes t^{r+s}+r\delta(r,-s)(X|Y)c+b_1sY\otimes t^s-b_2rX\otimes t^r.
\]

We are interested only in the nontwisted case, so we do not discuss the twisted case.

To define a symmetrizable affine Lie superalgebra abstractly define a {\em datum} to be a triple $(\EE,\Pi,p)$, consisting of: (1) A finite dimensional $\C$-vector space $\EE$ with a non-degenerate symmetric bilinear form $(\cdot ,\cdot):\EE\times\EE\to\C$. (2) A linearly independent subset $\Pi=\{\alpha_0,\alpha_1,\dots,\alpha_{n'{}'}\}$ of $\EE$; the $\alpha_i$ are called {\em simple roots}; $P=\Z\alpha_0\oplus\dots\oplus\Z\alpha_{n'{}'}$ the {\em root lattice}; $P_+=\Z_{\ge 0}\alpha_0\oplus\dots\oplus\Z_{\ge 0}\alpha_{n'{}'}$ the {\em positive root semilattice}; put $P_-=-P_+$. (3) A function $p:\Pi\onto\Z/2$; it extends uniquely to a group homomorphism $p:P\onto\Z/2$, called a {\em parity} function. Define the {\em Cartan algebra} $\HH=\EE^\ast=\Hom(\EE,\C)$ to be the linear dual of $\EE$. Identify an element $\nu\in\EE$ with $H_\nu\in\HH$ by $\mu(H_\nu)=(\mu,\nu)$ for all $\mu\in\EE$.

For a datum $(\EE,\Pi,p)$ define a Lie superalgebra $\wt{G}=\wt{G}(\EE,\Pi,p)$ by generators
\[
H\in\HH;\qquad E_i,\quad F_i\quad (0\le i\le n'{}');
\]
relations:
\[
[H,H']=0\qquad (H,\,H'\in\HH),\]
\[ [H,E_i]=\alpha_i(H)E_i,\qquad [H,F_i]=-\alpha_i(H)F_i,\]
\[ [E_i,F_j]=\delta(i,j)H_{\alpha_i};
\]
and parities
\[
p(E_i)=p(F_i)=p(\alpha_i),\qquad p(H)=0\quad(H\in\HH).
\]
The superalgebra $\wt{G}$ has a triangular decomposition $\wt{G}=\wt{\NN}^+\oplus\wt{\HH}\oplus\wt{\NN}^-$, where $\wt{\NN}^+$ is the free superalgebra with generators $E_i$, and $\wt{\NN}^-$ with $F_i$.

To a given datum $(\EE,\Pi,p)$ associate a partially ordered set $I(\EE,\Pi,p)$ of {\em admissible} Lie superalgebras $G^\#=\wt{G}/r^\#$, where $r^\#$ is an ideal of $\wt{G}$ which is {\em admissible}: $r^\#\cap\HH=\{0\}$. The partial order $>$ is defined by $\wt{G}/r_1^\#>\wt{G}/r_2^\#$ if $r_1^\#\subset r_2^\#$. Then $\wt{G}$ is the unique {\em top} element. Denote by $G$ the unique {\em bottom} element of $I(\EE,\Pi,p)$. It is called a {\em minimal admissible} Lie superalgebra. Note that $G_1^\#>G_2^\#$ iff there is a surjection $\Psi[G_1^\#,G_2^\#]:G_1^\#\to G_2^\#$ with $(H,E_i,F_i)\to(H,E_i,F_i)$. For any subset $B$ of $I(\EE,\Pi,p)$ define 
\[
V_BG^\#=\wt{G}/\cap_{G^\#\in B}\ker \Psi[\wt{G},G^\#].
\]
It is $>G^\#$ for all $G^\#\in B$. For any $\alpha\in\EE$ and $G^\#\in I(\EE,\Pi,p)$ put $G_\alpha^\#=\{X\in G^\#;\,[H,X]=\alpha(H)X,\,\forall H\in\HH\}$ and $\Phi[G^\#]=\{\alpha\in\EE-\{0\};\,\dim G_\alpha^\#\not=0\}$. The subspace $G_0^\#=\HH$, called the {\em Cartan subalgebra} of $G^\#$, is the same for all $G^\#$. Clearly $\Phi[G^\#]\subset P_+\cup\P_--\{0\}$, and $G_1^\#>G_2^\#$ implies $\Phi[G_1^\#]\supset\Phi[G_2^\#]$. Put $\Phi(\EE,\Pi,p)=\Phi[G]$.

In a Dynkin diagram associated with a datum $(\EE,\Pi,p)$ occur vertices labeled by $\alpha_i$, or simply $i$, $0\le i\le n'{}'$, and marked by\\

\smallskip
\n $\bigcirc$, called {\em white}, if $(\alpha_i,\alpha_i)\not=0$ and $p(\alpha_i)=0$,\\
$\begin{picture}(15, 15)(-5,-3)
{\color{gray}\put(0, 0){\circle*{10}}}
\put(5, -3){,}
\end{picture}
$ called {\em gray}, if $(\alpha_i,\alpha_i)=0$ and $p(\alpha_i)=1$,\\
$\begin{picture}(15, 15)(-5,-3)
{\color{black}\put(0, 0){\circle*{10}}}
\put(5, -3){,}
\end{picture}
$ called {\em black}, if $(\alpha_i,\alpha_i)\not=0$ and $p(\alpha_i)=1$.\\
\\
We are interested only in Dynkin diagrams whose vertices are white and gray. 

There is no edge between the $i$th and $j$th vertices if $(\alpha_i,\alpha_j)=0$. There is an edge\\

\smallskip
\n $
\begin{picture}(50, 15)(-5, 7)
\put(0, 10){\circle{10}}
\put(5, 10){\line(1, 0){30}}
\put(40, 10){\circle{10}}
\put(-2, -2){\tiny$i$}
\put(38, -2){\tiny$j$}
\end{picture}$ if $(\alpha_i,\alpha_i)=(\alpha_j,\alpha_j)=-2(\alpha_i,\alpha_j)\not=0$,\\
$
\begin{picture}(50, 20)(-5,-3)
{\color{gray}\put(0, 0){\circle*{10}}}
\put(5, 0){\line(1, 0){35}}
\put(35, -4){\large$\times$}

\put(30, -6){\tiny$x$}

\put(-2, -12){\tiny$i$}
\put(38, -12){\tiny$j$}
\end{picture}$ if $(\alpha_i,\alpha_j)\not=0$, and also $x=-(\alpha_j,\alpha_j)/2$ if $(\alpha_j,\alpha_j)\not=0$.\\

\smallskip
\n Here {\large${\times}$} can be white or gray. The Dynkin diagram of interest to us is of type $(\Aa)^{(1)}$.

Dynkin diagram of type $(\Aa)^{(1)} $: 
\begin{figure}[h]
\begin{picture}(200, 60)(40,-20)
\put(50, 0){\circle{10}}
\put(50, -10){\tiny$\alpha_1$}

\put(55, 0){\line(1, 0){20}}
\put(80, 0){\circle{10}}
\put(80, -10){\tiny$\alpha_2$}

\put(85, 0){\line(1, 0){10}}
\put(95, -1){...}
\put(105, 0){\line(1, 0){10}}
 \put(120, 0){\circle{10}}
\put(115, -10){\tiny$\alpha_{m-1}$}

\put(125, 0){\line(1, 0){20}}
{\color{gray}\put(150, 0){\circle*{10}}}
\put(150, -10){\tiny$\alpha_m$}

\put(155, 0){\line(1, 0){20}}
\put(180, 0){\circle{10}}
\put(175, -10){\tiny$\alpha_{m+1}$}

\put(185, 0){\line(1, 0){10}}
\put(195, -1){...}
\put(205, 0){\line(1, 0){10}}
\put(220, 0){\circle{10}}
\put(220, -10){\tiny$\alpha_{n''}$}

{\color{gray}\put(136, 40){\circle*{10}}}
\put(135, 50){\tiny$\alpha_0$}

\put(130, 38){\line(-2, -1){75}}
\put(140, 38){\line(2, -1){75}}
\end{picture}
\caption{Affine $\sl(m,n)$}
\label{fig:even}
\end{figure}

The superscript $(1)$ mean nontwisted, and we omit it from now on. The Dynkin diagram of type $(\Aa)$ is determined by $m\ge1$ and $n'>m$, $n'\ge 3$. We put $n'{}'=n'-1$, and $n=n'-m$ (our ($n'=m+n$, $n'{}'=n'-1$, $m$) are $(N,n,N-n)$ in \cite[\S1.5]{y99}). The vertices labeled by $m$ and $0$ in the diagram are gray.

From now on we consider only those $(\EE,\Pi,p)$ of type $(\Aa)$. 

Let $\EE^{\ex}$ (``$\EE$-extended") be an $(n'+2)$-dimensional $\C$-vector space with a nondegenerate bilinear symmetric form $(\cdot ,\cdot)$ and a basis $(\ove_1,\dots,\ove_{n'},\delta,\Lambda_0)$ (whose elements are named the ``fundamental elements of $(\EE,\Pi,p)$)" satisfying
\[
(\ove_i,\ove_j)=\delta(i,j)\ovd_i\quad(\ovd_i=\pm1),\qquad(\ove_i,\delta)=(\delta,\delta)=(\Lambda_0,\Lambda_0)=0,\qquad(\delta,\Lambda_0)=1.
\]
Write $(\Aa)^g$ for $(\Aa)$ if $\sum_{1\le i\le n'}\ovd_i\not=0$, and $(\Aa)^b$ for $(\Aa)$ if $\sum_{1\le i\le n'}\ovd_i=0$ ($g$ = good, $b$ = bad). Put $\theta=\sum_{1\le i\le n'}\ovd_i\ove_i$. Define $\EE$ to be $\EE^{\ex}$ if $(\Aa)^b$, and $\{x\in\EE^{\ex};\,(x,\theta)=0\}$ if $(\Aa)^g$. Then $(\cdot,\cdot)$ restricts to a nondegenerate symmetric form on $\EE$. The vertices in the Dynkin diagram are labeled by the roots $\alpha_i=\ove_i-\ove_{i+1}$ $(1\le i< n')$, and $\alpha_0=\ove_{n'}-\ove_1$.

The Lie superalgebra $G=G(\EE,\Pi,p)$ of type $(\Aa)$ is called of {\em affine} type $(\Aa)$.

The infinite dimensional symmetrizable minimal admissible Lie superalgebras of finite growth are parametrized in \cite{vdl89}. They are the affine Lie superalgebras listed at \cite[\S1.5]{y99}. We shall be concerned here only with those of type $(\Aa)$ ($=(\Aa)^{(1)}$).

The diagram $(\Aa)^g$ corresponds to $\sl(m,n)=A(m-1,n-1)$, $m\not=n$ (in fact, decorated by a superscript $(1)$, to indicate the non-twisted form). The center of $\sl(m,m)$ is $\C I_{2m}$. Put $A(m-1,m-1)=\sl(m,m)/\C I_{2m}$. Note that $\sl(m,m)$ and $A(m-1,m-1)$ are not minimal admissible Lie superalgebras since their simple roots are linearly dependent. The simple roots of $\gl(m,m)$ are linearly independent. Let $(\sl(m,m)^{(1)})^{\HH}$ be the ${sub}$algebra $\sl(m,m)^{(1)}\oplus\C E_{1,1}$ of $\gl(m,m)^{(1)}$, where $E_{1,1}$ denotes the matrix whose only nonzero entry is 1 at the $(1,1)$ position. Put $(A(m-1,m-1)^{(1)})^{\HH}$ for the quotient $(\sl(m,m)^{(1)})^{\HH}/(\oplus_{k\not=0}\C I_{2m}\otimes t^k)$. It is a minimal admissible Lie superalgebra.

For our $(\EE,\Pi,p)$ of type $(\Aa)$ we have $\dim_{\C}G_\alpha=1$ for $\alpha\in\Phi[G]-\Z\delta$, $G=G(\EE,\Pi,p)$.

An admissible Lie superalgebra $G^\#>G$ is called {\em affine admissible} if $\Phi[G^\#]=\Phi[G]$ and $\dim_{\C}G_\alpha^\#=1$ for each $\alpha\in\Phi[G]-\Z\delta$. Write $\AI=\AI(\EE,\Pi,p)$ for the set of affine admissible Lie superalgebras with datum $(\EE,\Pi,p)$. Then $G_{\AI}^\#=\wt{G}/\cap_{G^\#\in \AI}\ker\Psi[\wt{G},G^\#]$ is the unique maximal affine-admissible Lie superalgebra in $\AI$.

Let $\rho\in\EE$ be a vector satisfying $(\rho,\alpha_i)=\frac12(\alpha_i,\alpha_i)$ for all $\alpha_i\in\Pi$ (\cite[Proposition 1.2.2]{y99}). If $(\delta,\rho)\not=0$ then $G^\#_{\AI}=G$. An example where $(\delta,\rho)=0$ is $G_1$ with Dynkin diagram $(\Aa)$ and $n'=4$, with parity given by $p(\alpha_1)=p(\alpha_3)=0$, $p(\alpha_0)=p(\alpha_2)=1$. Then $G_1=(A(1,1)^{(1)})^{\HH}\not=G_{1,\AI}^\#=(\sl(2,2)^{(1)})^{\HH}$, and $\dim(G_1)_{k\delta}=2\not=3=\dim_\C(G_{1,AJ}^\#)_{k\delta}$ for $k\not=0$. In fact, $G_{\AI}^\#$ is $G=\sl(m,n)^{(1)}$ in case $(\Aa)^g$ (thus $m\not=n$), and it is $(\sl(n'/2,n'/2)^{(1)})^{\HH}$ in case $(\Aa)^b$ (where $n'=2m$, $n=m$) (\cite[Theorem 3.5.1]{y99}).

\begin{thm}\label{T5.1} 
$($\cite[Theorem 4.1.1]{y99}$)$. The Lie superalgebra $G^\#_{\AI}$ of datum $(\EE,\Pi,p)$ of type $(\Aa)$ can also be defined by generators $H\in\HH$, $E_i$, $F_i$ $(0\le i\le n'{}')$, parities $p(H)=0$, $p(E_i)=p(F_i)=p(\alpha_i)=0$ $(i\not=0,\,m)$, $p(E_0)=p(F_0)=1=p(E_m)=p(F_m)$, and relations\\
$(\mathrm{S}1)$ $\qquad [H,H']=0,\qquad H,\,H'\in\HH;$\\
$(\mathrm{S}2)$ $\qquad [H,E_i]=\alpha_i(H)E_i,\qquad [H,F_i]=-\alpha_i(H)F_i;$\\
$(\mathrm{S}3)$ $\qquad [E_i,F_j]=\delta_{ij}H_{\alpha_i};$\\
$(\mathrm{S}4)(1)$ $\quad [E_i,F_j]=0 \,\,\,{if}\,\,\, \dist(\alpha_i,\alpha_j)\ge 2;$\\
$(\mathrm{S}4)(2)$ $\quad [E_0,F_0]=0=[E_m,E_m],\,\,\,{thus}\,\,\, E_0^2=0=E_m^2;$\\
$(\mathrm{S}4)(3)$ $\quad [E_i,[E_i,E_{i+1}]]=0=[E_i,[E_i,E_{i-1}]]\quad(0\not=i\not=m);$\\
$(\mathrm{S}4)(4)$ $\quad [[[E_{m-1},E_m],E_{m+1}],E_m]=0,\,\,\,{equivalently}\,\,\, [[[E_{m+1},E_m],E_{m-1}],E_m]=0;$\\
$(\mathrm{S}4)(4')$ $\quad [[[E_{n'{}'},E_0],E_1],E_0]=0,\,\,\,{equivalently}\,\,\, [[[E_1,E_0],E_{n'{}'}],E_0]=0;$\\
$(\mathrm{S}5)(1)-(\mathrm{S}5)(4'):$ same as $(\mathrm{S}4)(1)-(\mathrm{S}4)(4')$ with $F_j$ replacing $E_j$.
\end{thm}

\section{Affine quantum Lie superalgebras}
Finally we arrive to the description of the objects of interest in this work, the affine quantum Lie superalgebras and their defining relations, following \cite[\S6]{y99, y01, yy}. Here the Quantum-Serre relations (QS) replace the Serre relations (S) of Theorem \ref{T5.1}.

Let $\C(q)$ denote the field of rational functions in an indeterminate $q$. Denote by $\sigma$ the generator of $\Z/2$. Let $V$ be a $\Z/2$-graded $\C(q)$-algebra. It is a Lie $\C(q)$-superalgebra with the superbracket defined by linearity and
\[
[X,Y]=XY-(-1)^{p(X)p(Y)}YX
\]
on homogeneous elements $X$, $Y$ of $V$. Now $\Z/2$ acts on $V$ by $\sigma(X)=(-1)^{p(X)}X$ on homogeneous elements. Define the $\C(q)$-algebra $V^\sigma$, called the {\em extension} of $V$ by $\sigma$, to be $V\rtimes\Z/2=V\oplus\sigma V$, where $V\sigma=\sigma V$ and $\sigma X\sigma=\sigma(X)$. For $\Z/2$-graded $\C(q)$-algebras homomorphisms $\varphi:V\to W$, define the extension $\varphi^\sigma:\,V^\sigma\to W^\sigma$ of $\varphi$ by $\sigma$, by $\varphi^\sigma(X)=\varphi X$ ($X\in V$) and $\varphi^\sigma(\sigma)=\sigma$. It is an algebra homomorphism.

Let $(\EE,\Pi,p)$ be a datum. A quadruple $(\EE,\Pi,p,\Gamma)$ is a {\em lattice datum} if (a) $\Gamma$ is a {\em lattice} in $\EE$, namely $\Gamma$ is a $\Z$-span of a basis of $\EE$, (b) $\Pi\subset\Gamma$, (c) $(\gamma,\gamma')\in\Z$ for all $\gamma$, $\gamma'\in\Gamma$.

For a lattice datum $(\EE,\Pi=\{\alpha_0,\dots,\alpha_{n'{}'}\},p,\Gamma)$, define an associative $\Z/2$-graded $\C(q)$-algebra $\wt{U}_q=\wt{U}_q(\EE,\Pi,p,\Gamma)$ with 1, by generators
\[
K_\gamma\quad(\gamma\in\Gamma),\qquad E_i,\quad F_i\quad (0\le i\le n'{}'),
\]
parities
\[
p(K_\gamma)=0,\qquad p(E_0)=p(F_0)=p(E_m)=p(F_m)=1,\quad p(E_i)=p(F_i)=0\quad(i\not=0,\,m),
\]
and relations\\
$(\QS1)\qquad K_0=1,\qquad K_\gamma K_{\gamma'}=K_{\gamma+\gamma'}\qquad(\gamma,\,\gamma'\in\Gamma);$\\
$(\QS2)\qquad K_\gamma E_i K_\gamma^{-1}=q^{(\gamma,\alpha_i)}E_i,\qquad K_\gamma F_i K_\gamma^{-1}=q^{-(\gamma,\alpha_i)}F_i;$\\
$(\QS3)\qquad [E_i,F_j]=\delta_{ij}\frac{K_{\alpha_i}-K_{\alpha_i}^{-1}}{q_i-q_i^{-1}}\quad(0\le i,\,j\le n'{}').$\\
Recall from section 2 that $d_i=1$ ($1\le i\le m$) and $d_i=-1$ ($m<i\le n'$), and $d_0=d_{n'}$.

Note that the sub-Hopf-algebra of $U_q(\fg)$ of section 3 generated by the $e_i$, $f_i$, $q^{h_i}$, $\sigma$ for $1\le i\le n'{}'$ embeds naturally in the affine quantum Lie superalgebra $U_q^\sigma$ of this section by
\[
e_i\mapsto E_iK_{\alpha_i}^{-1},\quad f_i\mapsto K_{\alpha_i}E_i,\quad q^{h_i}\mapsto K_{\alpha_i}^{d_i},\quad \sigma\mapsto\sigma.
\]
This embedding is a homomorphism of Hopf algebras. Indeed we have $\frac{q^{h_i}-q^{-h_i}}{q-q^{-1}}=\frac{q^{d_ih_i}-q^{-d_ih_i}}{q_i-q_i^{-1}}$ as $q_i=q^{d_i}$. Note that the finite-type Schur-Weyl duality of Proposition \ref{P11.1} below is stated for $U_q^\sigma(\sl(m,n))$ to simplify the notation, but it applies to $U_q^\sigma$.

The extension $\wt{U}_q^\sigma=\wt{U}_q^\sigma(\EE,\Pi,p,\Gamma)$ of $\wt{U}_q$ by $\sigma$ has a {\em Hopf algebra} structure $(\wt{U}_q,\Delta,S,\varep)$. The {\em comultiplication} $\Delta$, {\em antipode} $S$, {\em counit} $\varep$ are defined by
\[
\Delta(\sigma)=\sigma\otimes\sigma,\qquad \Delta(K_\gamma)=K_\gamma\otimes K_\gamma,\]
\[
\Delta(E_i)=E_i\otimes 1+K_{\alpha_i}\sigma^{p(\alpha_i)}\otimes E_i,\qquad \Delta(F_i)=F_i\otimes K_{\alpha_i}^{-1}+\sigma^{p(\alpha_i)}\otimes F_i,\]
\[
S(\sigma)=\sigma,\qquad S(K_\gamma)=K_\gamma^{-1},\qquad S(E_i)=-K_{\alpha_i}^{-1}\sigma^{p(\alpha_i)}E_i,\qquad S(F_i)=-\sigma^{p(\alpha_i)}F_iK_{\alpha_i},\]
\[
\varep(\sigma)=1,\qquad \varep(K_\gamma)=1,\qquad\varep(E_i)=0,\qquad\varep(F_i)=0.
\]

Let $\wt{U}_q^+$ be the subalgebra with 1 of $\wt{U}_q$ generated by the $E_0,\dots,E_{n'{}'}$. It is in fact freely generated by these generators. The analogous statement holds for $\wt{U}_q^-$, generated by the $F_0,\dots,F_{n'{}'}$. The subalgebra with 1 of $\wt{U}_q$ generated by the $K_\gamma$ ($\gamma\in\Gamma$) is denoted by $T$. The $K_\gamma$ ($\gamma\in\Gamma$) make a basis of $T$, and there is a $\C(q)$-linear isomorphism
\[
\wt{U}_q^+\otimes_{\C(q)}T^\sigma\otimes_{\C(q)}\wt{U}_q^-\to\wt{U}_q^\sigma,\qquad X\otimes Z\otimes Y\mapsto XZY,\qquad T^\sigma=T\oplus\sigma T.
\]

Following Drinfeld, \cite[6.3]{y99} defines subspaces $I^+$ of $\wt{U}_q^+$ and $I^-$ of $\wt{U}_q^-$, and Hopf ideals $J^+=I^+T^\sigma\wt{U}_q^-$ and $J^-=\wt{U}_q^+T^\sigma I^-$ of $\wt{U}_q^\sigma$, and the Hopf algebra
\[
U_q^\sigma=U_q^\sigma(\EE,\Pi,p,\Gamma)=(U_q^\sigma,\Delta,S,\varep)=\wt{U}_q^\sigma/(J^++J^-).
\]
If $\pi:\wt{U}_q^\sigma\to U_q^\sigma$ is the quotient map, write $\sigma$, $K_\gamma$, $E_i$, $F_i$, $T$, $T^\sigma$, $\wt{U}_q^+$, $\wt{U}_q^-$ also for their images under $\pi$. Note that $\ker\pi|T^\sigma=\{0\}$, $\ker\pi|\wt{U}_q^+=I^+$, $\ker\pi|\wt{U}_q^-=I^-$. There is a $\C(q)$-linear isomorphism
\[
U_q^+\otimes_{\C(q)}T^\sigma\otimes_{\C(q)}U_q^-\xrightarrow{\sim}U_q^\sigma,\qquad X\otimes Z\otimes Y\mapsto XZY.
\]
\begin{rem} 
Specialization at $q=1$. Let $A=\C[q,q^{-1}]$ be the $\C$-subalgebra of the field $\C(q)$ generated by $q$ and $q^{-1}$. Let $U_A^\sigma$ be the $A$-algebra of $U_q^\sigma=U_q^\sigma(\EE,\Pi,p,\Gamma)$ generated by
\[
\sigma,\quad K_\gamma,\quad [K_\gamma]=\frac{K_\gamma-K_\gamma^{-1}}{q-q^{-1}}\quad (\gamma\in\Gamma),\quad E_i,\,\,\, F_i\,\,\,(0\le i\le n'{}').
\]
Let $T_A^\sigma$ (resp. $U_A^+$, $U_A^-$) be the subalgebra of $T^\sigma$ (resp. $U_q^+$, $U_q^-$) generated by $\sigma$, $K_\gamma$, $[K_\gamma]$ (resp. $E_i$, $F_i$). Note that if $\{\gamma(r);\,1\le r\le\dim\EE\}$ is a $\Z$-basis of $\Gamma$, then
\[
\{\sigma^{a(0)}\prod_{1\le r\le\dim\EE}K_{\gamma(r)}^{a(r)}[K_{\gamma(r)}]^{b(r)};\,a(0)\in\{0,1\};\,a(r),\,b(r)\in\Z_{\ge 0}\}
\]
is an $A$-basis of $T^\sigma$. Also there is an $A$-module isomorphism
\[
U_A^+\otimes_AT_A^\sigma\otimes_AU_A^-\xrightarrow{\sim}U_A^\sigma,\qquad X\otimes Z\otimes Y\mapsto XZY.
\]

Denote by $\C_1$ the field $\C$ regarded as a left $A$-module in which $q$ acts as 1. Define the $\C$-algebras $U_\C^\sigma=U_A^\sigma\otimes_A\C_1$, $T_\C^\sigma=T_A^\sigma\otimes_A\C_1$, $U_\C^+=U_A^+\otimes_A\C_1$, $U_\C^-=U_A^-\otimes_A\C_1$. There is a $\C$-linear isomorphism
\[
U_\C^+\otimes_\C T_\C^\sigma\otimes_\C U_\C^-\xrightarrow{\sim}U_\C^\sigma,\qquad X\otimes Z\otimes Y\mapsto XZY.
\]
Define
\[
{}'U_\C^\sigma=U_\C^\sigma/V,\qquad V=\langle K_\gamma\otimes_A1-1\otimes_A1;\,\gamma\in\Gamma\rangle;
\]
$V$ is a 2-sided ideal in $U_\C^\sigma$. Define $\pi_1:U_A^\sigma\to\,{}'U_\C^\sigma$ by $\pi_1(X)=X\otimes_A1+V$. Denote the images of $\sigma$, $[K_\gamma]$, $E_i$, $F_i$ under $\pi_1$ by $\sigma$, $H_\gamma$, $E_i$, $F_i$. Recall that an element $\nu\in\EE$ was identified with $H_\nu\in\HH=\EE^\ast$ by $\mu(H_\nu)=(\mu,\nu)$ for all $\mu\in\EE$. Then there is a unique Lie $\C$-superalgebra homomorphism $\wt\chi:\wt{G}(\EE,\Pi,p)\to\,{}'U_\C^\sigma$ with $\wt\chi(H_\gamma)=H_\gamma$ ($\gamma\in\Gamma$), $\wt\chi(E_i)=E_i$, $\wt\chi(F_i)=F_i$ ($0\le i\le n'{}'$). Define an admissible Lie superalgebra $G^\Gamma=G^\Gamma(\EE,\Pi,p)$ in $I(\EE,\Pi,p)$ by $\wt{G}(\EE,\Pi,p)/\ker\wt\chi$. Denote by $\chi:G^\Gamma\to\,{}'U_\C^\sigma$ the Lie superalgebra monomorphism obtained from $\chi$. Let $U(G^\Gamma)$ be the universal enveloping superalgebra of $G^\Gamma$. Denote by $\Xi:U(G^\Gamma)\to\,{}'U_\C^\sigma$ the surjection with $\Xi|G^\Gamma=\chi$. Then the extension $U(G^\Gamma)^\sigma$ of $U(G^\Gamma)$ by $\sigma$ has a Hopf $\C$-algebra structure, the extension $\Xi^\sigma$ of $\Xi$ by $\sigma$ is a Hopf $\C$-algebra surjection, which in fact is an isomorphism (\cite[Lemma 6.6.1]{y99}). In conclusion, the specialization of $U_A^\sigma$ at $q=1$, ${}'U_\C^\sigma$, is $U(G^\Gamma)^\sigma$, where $G^\Gamma=\wt{G}(\EE,\Pi,p)/\ker\wt\chi\in I(\EE,\Pi,p)$.
\end{rem}

For datum $(\EE,\Pi,p)$ of affine $(\Aa)$ type, fix a lattice datum $(\EE,\Pi,p,\Gamma)$ by\\
$\Gamma=\Z\ove_1\oplus\dots\oplus\Z\ove_{n'}\oplus\Z\delta\oplus\Z\Lambda_0$ if $(\Aa)^b$, namely $\sum_{1\le i\le n'}\ovd_i=0$;\\
$\Gamma=P\oplus\Z\Lambda_0$, $P=\Z\alpha_0\oplus\dots\oplus\Z\alpha_{n'{}'}$ if $(\Aa)^g$, namely $\sum_{1\le i\le n'}\ovd_i\not=0$,\\
and say that $(\EE,\Pi,p,\Gamma)$ is of {\em affine} $(\Aa)$-{\em type}.

For $\alpha\in P$, put $(U_q^\sigma)_\alpha=\{X\in U_q^\sigma;\,K_\gamma XK_\gamma^{-1}=q^{(\alpha,\gamma)}X,\quad\forall\gamma\in\Gamma\}$.

For $X_\alpha\in(U_q^\sigma)_\alpha$, $X_\beta\in(U_q^\sigma)_\beta$, put
\[
\br X_\alpha,X_\beta\ebr=X_\alpha X_\beta-(-1)^{p(X_\alpha)p(X_\beta)}q^{-(\alpha,\beta)}X_\beta X_\alpha.
\]
Recall that $[X,Y]=XY-(-1)^{p(X)p(Y)}YX$ for homogeneous $X$, $Y$.
\begin{prop}\label{P6.1} 
$($\cite[Theorem 6.8.2]{y99}$)$. Let $(\EE,\Pi,p,\Gamma)$ be a lattice datum of affine $(\Aa)$-type. Then the $\C(q)$-algebra $U_q(\EE,\Pi,p,\Gamma)$ can also be defined by generators $K_\gamma$ $(\gamma\in\Gamma)$, $E_i$, $F_i$ $(0\le i\le n'{}')$, parities $p(K_\gamma)=0$, $p(E_i)=p(F_i)=p(\alpha_i)=0$ $(i\not=0,\,m)$, $p(E_0)=p(F_0)=1=p(E_m)=p(F_m)$, and the quantum Serre relations $(\QS1)$, $(\QS2)$,
$(\QS3)$, $(\QS4)(a)$ and $(\QS5)(a)$, $1\le a\le 4$, where $(\QS5)(a)$ are obtained from $(\QS4)(a)$ on replacing $E_j$ by $F_j$, and$:$\\
$(\QS4)(1)\quad[E_i,E_j]=0$ if $\dist(\alpha_i,\alpha_j)\ge 2$ $($i.e., $i\not=j$ and $(\alpha_i,\alpha_j)=0);$\\
$(\QS4)(2)\quad[E_0,E_0]=0=[E_m,E_m]$, i.e. $E_0^2=0=E_m^2;$\\
$(\QS4)(3)\quad\br E_i,\br E_i,E_{i\pm 1}\ebr\ebr=0,$ $0\not=i\not=m;$ i.e. $E_i^2E_{i\pm 1}-(q+q^{-1})E_iE_{i\pm 1}E_i+E_{i\pm 1}E_i^2=0;$
$(\QS4)(4)$ $\quad [ \br\br E_{m-1},E_m\ebr,E_{m+1}\ebr,E_m]=0,\,\,\,{equivalently}\,\,\, [\br\br E_{m+1},E_m\ebr,E_{m-1}\ebr,E_m]=0;$\\
$(\QS4)(4')$ $\quad [ \br\br E_{n'{}'},E_0\ebr,E_1\ebr,E_0]=0,\,\,\,{equivalently}\,\,\, [ \br\br E_1,E_0\ebr,E_{n'{}'}\ebr,E_0]=0;$\\
$(\QS5)(1)-(\QS5)(4'):$ same as $(\QS4)(1)-(\QS4)(4')$ with $F_j$ replacing $E_j$.
\end{prop}
We assume in Proposition \ref{P6.1} that $m\not=n$. When $m=n$ there are additional relations (see \cite[Theorem 8.4.3]{y99}). Note that $(\QS4)(1)$ asserts $E_0E_m+E_mE_0=0$, and $E_iE_j-E_jE_i=0$ in the other cases.

\section{Hecke algebra}\label{7}

Next we proceed to introduce an action of the affine Iwahori-Hecke algebra, denoted $H^a_d(q^2)$, on $V^{\otimes d}$, via the theory of $R$-operators, developed from Drinfeld's and Jimbo's solution of the quantum Yang-Baxter equation.

By the Hecke algebra in the theory of admissible representations one usually means the convolution algebra $H$ of complex valued compactly supported measures on a local reductive group. It suffices to consider here $G(F)$, the group of $F$-points of a reductive connected group $G$ over $F$, $G=\GL(d)$ in our case, where $F$ is a local non-Archimedean field. Fixing a Haar measure $dg$, and noting that a measure in $H$ has the form $f\,dg$ where $f:G(F)\to\C$ is compactly supported and smooth (biinvariant under a compact open subgroup $K'$ of $G(F)$), one identifies $H$ with $\cup_{K'}C_c(K'\bs G(F)/K')$. The {\em spherical Hecke algebra} $H_K=C_c(K\bs G(F)/K)$, where $K=G(\OO)$ is the hyperspecial maximal compact subgroup of our $G(F)$, $\OO$ being the ring of integers of $F$, is commutative, and can be studied by means of the Satake transform. When $K'$ is taken to be an Iwahori subgroup $I\subset G(\OO)$, the pullback of $B(\F_q)$, the upper triangular Borel subgroup of $G(\F_q)$, under the reduction $G(\OO)\to G(\F_q)$, obtained from the reduction $\OO\to\F_q=\OO/(\upi)$ modulo the maximal ideal $(\upi)$ in the local ring $\OO$, the convolution algebra $H_I=C_c(I\bs G(F)/I)$ is called the {\em affine Iwahori-Hecke algebra}. The structure of this algebra, of great importance in the study of admissible representations of the group of points over a local non-Archimedean field $F$ of a reductive connected group $G$, was studied by Iwahori and Matsumoto \cite{im65}, who gave a presentation in terms of generators and relations. This presentation depends on the (residual) cardinality $q$ of $\F_q$, but the isomorphism class of the algebra $\H_I$ need not be, as specified in Proposition \ref{P7.2}.

\subsection{Bernstein presentation}
A useful presentation, better reflecting the structure of the group $G(F)$, was given by J. Bernstein; see \cite{hkp10} for a clear exposition. Bernstein's presentation exhibits $H_I$ as composed from the {\em finite Iwahori-Hecke algebra} $C_c(I\bs K/I)$, $K=G(\OO)$, which can be presented as generated  by $T_1,\dots,\,T_d$, subject to the relations: (1)  $T_iT_j=T_jT_i$ if $|i-j|>1$ (commutation relations); (2) $T_iT_{i+1}T_i=T_{i+1}T_iT_{i+1}$  (braid relations); (3) $(T_i+1)(T_i-q^2)=0$ (quadratic relations), and the commutative spherical algebra $R=C_c(A\cap I\bs A)$, where $A$ denotes the diagonal subgroup of $G(F)$, which is generated by the commuting elements $\theta_\lambda$, $\lambda\in P^{\vee}\simeq A/A\cap I$. This is the same as the lattice  $P^{\vee}$ of section \ref{2} in our case of $G=\GL(d)$. The $T_i$ can be viewed as normalized characteristic functions of the double cosets $Is_iI$, where $s_i$ ($1\le i<d$) are transpositions generating the Weyl group $W=I\bs K/I$, which is the symmetric group on $d$ letters: $=S_d$, in our case of $\GL(d)$, where the $s_i$ are the reflections $(i,i+1)$. Bernstein observed that to determine $H$, the commutation relations between the $T_i$ and the $\theta_\lambda$ need to be specified. Bernstein's presentation of $H_I$ consists then of generators $T_i=T_{\alpha_i}$, where $\alpha_i$ range over a base of simple roots, in our case $1\le i<d$, satisfying the commutation relations, the braid relations, and the quadratic relations, and the commuting generators $\theta_\lambda$, $\lambda\in P$, satisfying the Bernstein relations
\[
\theta_\lambda T_i-T_i\theta_{s_i\lambda}=(q^2-1)\frac{\theta_\lambda-\theta_{s_i\lambda}}{1-\theta_{-\alpha_i^\vee}}.
\]
Here $s_i$ are the reflections in the Weyl group associated with the coroots $\alpha_i^\vee\in P$. Note that the quadratic relations imply $T_i^{-1}=q^{-2}T_i+(q^{-2}-1)$, namely that the $T_i$ are invertible. 

In our case of $\GL(d)$, the lattice $P\simeq\Z^d$ is spanned by $\varep_i=(0,\dots,0,1,0,\dots,0)$, 1 in the $i$th place, $1\le i\le d$. The simple coroots are $\alpha_i^\vee=\varep_i-\varep_{i+1}$, and the corresponding reflections $s_i$ interchange $\varep_i$ and $\varep_{i+1}$, and fix the other $\varep_j$. Thus the lattice $\{\theta_\lambda;\,\lambda\in P\}$ is generated by $y_i=\theta_{-\varep_i}$. To write the Bernstein relation for $\lambda=-\varep_i$, note that $s_i\lambda=-\varep_{i+1}$, so $y_i=\theta_\lambda$, $\theta_{s_i(\lambda)}=y_{i+1}$, $\theta_{-\alpha_i^\vee}=y_iy_{i+1}^{-1}$. Hence the relation in this case is
\[
y_iT_i-T_iy_{i+1}=(q^2-1)\frac{y_i-y_{i+1}}{1-y_iy_{i+1}^{-1}}=-(q^2-1)y_{i+1}.
\]
So
\[
y_iT_i=(T_i-(q^2-1))y_{i+1}=T_i^{-1}q^2y_{i+1},\quad\mathrm{or}\quad T_iy_iT_i=q^2y_{i+1}.
\]
Normalizing $T_i$ by putting $\wh T_i=q^{-1}T_i$, the relation becomes $\wh T_iy_i\wh T_i=y_{i+1}$. In summary:
\begin{dfn}\label{D7.1} 
Fix $d\ge 1$ and $q\in\C^\times$ which is not a root of $1$. The {\em affine Iwahori-Hecke} $($in short: {\em Hecke}$)$ {\em algebra} $H^a_d(q^2)$ is the associative algebra over $\C(q)$ with $1$ generated by $\wh T_i$ $(1\le i<d)$ and $y_j^{\pm 1}$ $(1\le j\le d)$, subject to the relations
\[
\wh T_i\wh T_{i+1}\wh T_i=\wh T_{i+1}\wh T_i\wh T_{i+1};\qquad \wh T_i\wh T_j=\wh T_j\wh T_i \quad(|i-j|>1);
\]
\[
(\wh T_i+q^{-1})(\wh T_i-q)=0, \qquad\mathrm{or}\,\,\,\wh T_i^2-(q-q^{-1})\wh T_i-1=0;
\]
\[
y_jy_j^{-1}=1=y_j^{-1}y_j;\qquad y_jy_k=y_ky_j;\qquad y_j\wh T_i=\wh T_iy_j\,\,\,\mathrm{if}\,\,\,j\not=i,\,i+1;
\]
\[
\wh T_iy_i\wh T_i=y_{i+1}.
\]
\end{dfn}
The associative subalgebra with $1$ over $\Q(q)$ generated by the $\wh T_i$ $(1\le i<d)$ is called the {\em finite Hecke algebra} $H_d(q^2)$. Then we have
\begin{prop}\label{P7.1}
$H_d(q^2)\into H^a_d(q^2)$, and $\C(q)[y_1^{\pm 1},\dots,y_d^{\pm 1}]\otimes H_d(q^2)\to H^a_d(q^2)$ is an isomorphism of vector spaces.
\end{prop}
By Maschke theorem, the group algebra of a finite group $W$ over a field $F$ is semisimple precisely when the characteristic of $F$ does not divide the order of $W$. Consider the Hecke algebra $H(q^2)$ defined over the field $\Q(q)$ (instead of $\C(q)$), associated with any Coxeter system $(W,S)$ of any type, not just type $A$, where the Weyl group $W$ is $S_d$. By \cite{gu89}, $H(q^2)$ is semisimple if and only if $q^2$ is not a root of the Poincar\'e polynomial $P(q)=\sum_{w\in W}q^{\ell(w)}$, in particular when $q^2$ is not a root of 1. Here $\ell:W\to\Z_{\ge 0}$ is the length function on $(W,S)$. 
\begin{prop}\label{P7.2}
The Hecke algebra $H(q^2)$ and the group algebra $\Q(q)W$ of the Weyl group $W$ over the field $\Q(q)$ are isomorphic whenever $H(q^2)$ is semisimple.
\end{prop}
\rem This is proven in \cite{l81}. If $q$ is indeterminate, the isomorphism $H(q^2)\simeq\Q(q)W$ is already in \cite[Ex. 26, 27, p. 56]{b68}. See \cite[chapter 13]{g93} for an exposition in type A. Note that the $f$ in the proof of \cite[Lemma 2.12]{mo03} does not preserve the braid relations.
\section{Parabolic induction}\label{8}
The natural embedding $S_d\times S_b\into S_{d+b}$, where $S_d$ is the symmetric group on the letters $t_1,\dots,t_d$, and $S_b$ on $t_{d+1},\dots,t_{d+b}$, extends to the functor of (normalized) induction of admissible representations. Considering only the Iwahori-unramified case, we have:

\begin{prop}\label{P8.1}
There exists a unique homomorphism $\phi(d,b):\,H_d^a(q^2)\otimes H_b^a(q^2)\to H_{d+b}^a(q^2)$ of Hecke algebras which maps $\wh T_i\otimes 1\mapsto\wh T_i$, $y_j\otimes 1\mapsto y_j$ $(1\le i<d$, $1\le j\le d)$, $1\otimes\wh T_i\mapsto\wh T_{i+d}$, $1\otimes y_j\mapsto y_{j+d}$ $(1\le i<b$, $1\le j\le b)$. The restriction of $\phi(d,b)$ to $H_d(q^2)\otimes H_b(q^2)$ defines a homomorphism into $H_{d+b}(q^2)$.
\end{prop}

Let $M_i$ be a right $H_{d_i}^a(q^2)$-module $(i=1,\,2)$. Then $M_1\otimes M_2$, their outer tensor product, is an $H_{d_1}^a(q^2)\otimes H_{d_2}^a(q^2)$-module. The induced $H_{d_1+d_2}^a(q^2)$-module $I^a(M_1,M_2)$, studied by Bernstein and Zelevinsky \cite{bz76}, \cite{bz77}, \cite{ze80}, also named the Zelevinsky tensor product of $M_1$ and $M_2$, is defined by
\[
I^a(M_1,M_2)=\ind_{H_{d_1}^a(q^2)\otimes H_{d_2}^a(q^2)}^{H_{d_1+d_2}^a(q^2)}(M_1\otimes M_2)=(M_1\otimes M_2)\otimes_{H_{d_1}^a(q^2)\otimes H_{d_2}(q^2)}H_{d_1+d_2}^a(q^2).
\]
Up to a canonical isomorphism this induction is associative, and satisfies the pentagon axion (for a product of four objects). The same holds with the superscript $a$ removed, for finite Hecke algebras.

Let $M$ be an $H_d^a(q^2)$-module. By $M|H_d(q^2)$ we mean $M$ regarded as an $H_d(q^2)$-module by restriction.

\begin{prop}\label{P8.2}
Let $M_i$ be a finite dimensional $H_{d_i}^a(q^2)$-module $(i=1,2)$. Then there is a natural isomorphism $I^a(M_1,M_2)|H_{d_1+d_2}(q^2)\simeq I(M_1|H_{d_1}(q^2),M_2|H_{d_2}(q^2))$.
\end{prop}
\begin{proof}
The natural map $I^a(M_1,M_2)|H_{d_1+d_2}(q^2)\to I(M_1|H_{d_1}(q^2),M_2|H_{d_2}(q^2))$, $(m_1\otimes m_2)\otimes h\mapsto (m_1\otimes m_2)\otimes h$ ($m_i\in M_i$, $h\in H_{d_1+d_2}(q^2)$), is well-defined surjective homomorphism of $H_{d_1+d_2}(q^2)$-modules. By Proposition \ref{P7.1}, the rank of $H_{d_1+d_2}^a(q^2)$ as an $H_{d_1}^a(q^2)\otimes H_{d_2}^a(q^2)$-module is equal to the rank of $H_{d_1+d_2}(q^2)$ as an $H_{d_1}(q^2)\otimes H_{d_2}(q^2)$-module. Hence $\dim_{\C}I^a(M_1,M_2)=\dim_{\C}I(M_1,M_2)$.
\end{proof}

The theory of unramified representations suggests a family of universal $H_d^a(q^2)$-modules $M_c=H_d^a(q^2)/H_c$, where $c=(c_1,\dots,c_d)\in\C^{\times d}$ and $H_c$ is the right ideal in $H_d^a(q^2)$ generated by $y_j-c_j$ ($1\le j\le d$). It is part of the Langlands-Zelevinsky classification \cite{ze80} that

\begin{prop}\label{P8.3}
$(a)$ Every finite dimensional irreducible $H_d^a(q^2)$-module is isomorphic to a quotient of some $M_c$.\\
$(b)$ For all $c\in \C^{\times d}$, $M_c$ is isomorphic as an $H_d(q^2)$-module to the right regular representation.\\
$(c)$ $M_c$ is reducible as an $H_d^a(q^2)$-module iff $c_j=q^2c_k$ for some $j$, $k$.
\end{prop}

Some representations of $H_d^a(q^2)$ can be lifted from those of $H_d(q^2)$:
\begin{prop}\label{P8.4}
For each $z\in \C^\times$ there is a unique homomorphism $\ev_z:H_d^a(q^2)\to H_d(q^2)$ which is the identity on $H_d(q^2)\subset H_d^a(q^2)$ $($thus $\ev_z(\wh T_i)=\wh T_i$ $(1\le i<d))$ and maps $y_1$ to $z$. Moreover, $\ev_z(y_j)=z\wh T_{j-1}\wh T_{j-2}\dots\wh T_2\wh T_1^2\wh T_2\dots\wh T_{j-1}$ $(1\le j\le d)$.
\end{prop}

Let $M$ be any $H_d(q^2)$-module. Pulling back $M$ by $ev_z$ gives an $H_d^a(q^2)$-module $M(z)$ which is isomorphic to $M$ as an $H_d(q^2)$-module.

\section{Yang-Baxter equations}\label{9}
Quantum algebras were developed by Drinfeld \cite{d85}, \cite{d86} and Jimbo \cite{j86}. In particular \cite{d86} introduced quasi triangular Hopf algebras. This is a Hopf algebra $\AA$ with an element $R\in\AA\wh\otimes\AA$ satisfying (recall that $\Delta$ is the comultiplication, and put $\Delta'=\tau\Delta$, $\tau(x\otimes y)=y\otimes x$)
\[
\Delta'(x)=R\Delta(x)R^{-1}\quad\forall x\in\AA,
\]
\[(\Delta\otimes\id)R=R^{13}R^{23},\quad (\id\otimes\Delta)R=R^{13}R^{12},
\]
where $R^{12}$, $R^{23}$, $R^{13}$ are explicitly defined in Theorem \ref{T10.1} below. The element $R$ satisfies the Yang-Baxter (YB) equation; it is called {\em the universal $R$-matrix}. Construction of the quasi triangular Hopf algebra is based in \cite{d86} on the notion of the quantum double: the quantum double $W(\AA)$ of a Hopf algebra $\AA$ is a quasi triangular Hopf algebra isomorphic to $\AA\otimes\AA'$ as a vector space, with the canonical $R$-matrix $R=\sum_ie_i\otimes e^i$, where $\{e_i\}$ and $\{e^i\}$ are dual bases in $\AA$ and its dual $\AA'$. For any quantum algebra $U_q(\fg)$, which is the Drinfeld-Jimbo deformation of a Kac-Moody algebra $\fg$, there is a surjection to $U_q(\fg)$ from the quantum double of the corresponding Borel subalgebra: $W(U_q(\fb_+))\to U_q(\fg)$. Thus any quantum algebra $U_q(\fg)$ is a quasi triangular Hopf algebra. The problem is to give an explicit expression to the universal $R$-matrix directly in terms of $U_q(\fg)$. The implicit form of such an expression was given by \cite{d86}. 

For quantum superalgebras: $q$-deformations of finite dimensional contragredient Lie superalgebras, such a formula is given in \cite{kt91}. Let us recall the universal $R$-matrix in our case of $U_q(\gl(m,n))$. This $R$ will be used to construct an action of the Hecke algebra $H_d(q^2)$ on $V^{\otimes d}$, which commutes with the action $\rho_d$ of $U_q(\gl(m,n))$ on $V^{\otimes d}$.

\section{Universal $R$-matrix}\label{10}
Let $\tau$ be the (bilinear) involution $\tau(x\otimes y)=(-1)^{p(x)p(y)}y\otimes x$ (for homogeneous elements $x$ and $y$) of $U_q(\gl(m,n))\otimes U_q(\gl(m,n))$. Define the opposite comultiplication by $\Delta'=\tau\Delta$.
\begin{thm}\label{T10.1}
$($\cite{kt91}$)$ There is a unique invertible solution 
\[
R=\sum_ix_i\otimes y_i\in U_q(\gl(m,n))\wh\otimes U_q(\gl(m,n))
\]
of parity $0$ $($where $\wh\otimes$ means the completion$)$ of the equations
\[
\Delta'(x)=R\Delta(x)R^{-1}\quad\forall x\in U_q(\gl(m,n)),
\]
\[(\Delta\otimes\id)R=R^{13}R^{23},\quad (\id\otimes\Delta)R=R^{13}R^{12},
\]
where $R^{12}=\sum_ix_i\otimes y_i\otimes 1$, $R^{23}=\sum_i1\otimes x_i\otimes y_i$, $R^{13}=\sum_ix_i\otimes 1\otimes y_i$.
\end{thm}

This $R$ is called the universal $R$-matrix. The space $V$ is $\Z/2$-graded, $=V_{\0}\oplus V_{\1}$, and $\ove_i$, $(1\le i\le n')$, makes a basis with $V_{\0}=\oplus_{1\le i\le m}\C\ove_i$, $V_{\1}=\oplus_{m<i\le n'}\C\ove_i$. Then $p(\ove_i)$ is 0 if $1\le i\le m$, and 1 if $m<i\le n'$. Applying $R$ to $V\otimes V$ relative to the basis $\{\ove_i\otimes \ove_j;\,i,\,j\in I\}$, $I=\{1,\dots,n'\}$, the matrix $R$ in $\End(V\otimes V)$ is given by
\[
R=\sum_{1\le i\le n'}q^{(-1)^{p(\ove_i)}}E_{i,i}\otimes E_{i,i}
+\sum_{1\le i\not=j\le n'}E_{i,i}\otimes E_{j,j}
+\sum_{1\le i<j\le n'}(-1)^{p(\ove_i)}(q-q^{-1})E_{j,i}\otimes E_{i,j}.
\]
Here $\rho(E_{ij})\ove_k=\delta(j,k)\ove_i$. The action of $\End(V\otimes V)$ on $V\otimes V$ is $\Z/2$-graded, namely for homogeneous elements 
\[
X\otimes Y\in\End(V\otimes V)=\End(V)\otimes\End(V)
\]
and $u\otimes w\in V\otimes V$, we have 
\[
(X\otimes Y)(u\otimes w)=(-1)^{p(Y)p(u)}Xu\otimes Yw.
\]

The product of tensors is given as 
\[
(X_1\otimes X_2)(Y_1\otimes Y_2)=(-1)^{p(X_2)p(Y_1)}X_1Y_1\otimes X_2Y_2
\]
for homogeneous $X_1\otimes X_2$, $Y_1\otimes Y_2\in\End(V\otimes V)$. Define 
\[
\sigma:V\otimes V\to V\otimes V\,\,\,\mathrm{by}\,\,\, \sigma(u\otimes w)=(-1)^{p(u)p(w)}w\otimes u.
\]

Then $\check R=\sigma R$ is given by
\[
\check R=\sum_{1\le i\le n'}(-1)^{p(\ove_i)}q^{(-1)^{p(\ove_i)}}E_{i,i}\otimes E_{i,i}
+\sum_{1\le i\not=j\le n'}(-1)^{p(\ove_i)}E_{j,i}\otimes E_{i,j}+\sum_{1\le i<j\le n'}(q-q^{-1})E_{i,i}\otimes E_{j,j}.
\]

By direct computation we check that
\[
\check R^2=(q-q^{-1})\check R+I.
\]

For each $j$ ($1\le j<d$) put $\check R_j=\id_V^{\otimes(j-1)}\otimes\check R\otimes\id_V^{\otimes(d-j-1)}\in\End(V^{\otimes d})$, where $\check R$ operates on the $(j,j+1)$ factors. We have $(\check R_i+q^{-1})(\check R_i-q)=0$, and one checks (see \cite[Prop. 2.7]{mo03}, \cite[Thm 2.1]{mi06}):
\begin{prop}\label{P10.2}
$(1)$ The $\check R_j$ satisfy the commutation relations $\check R_i\check R_j=\check R_j\check R_i$ if $|i-j|\ge 2;$ and the braid relations $\check R_i\check R_{i+1}\check R_i=\check R_{i+1}\check R_i\check R_{i+1}$ $(1\le i<d)$, so they define an action $\pi_d$ of $H_d(q^2)$ on $V^{\otimes d}$. $(2)$ This action $\pi_d$ commutes with the natural action $\rho_d$ of $U_q(\gl(m,n))$ on $V^{\otimes d}$, namely $\check R_j\in\End_{(\rho_d,U_q(\gl(m,n)))}(V^{\otimes d})$ for all $j$ $(1\le j<d)$.
\end{prop}
For (2) see \cite[Prop. 2.9]{mo03}, \cite[Prop. 4.1, 4.2]{mi06}. Moreover, \cite[Thm 3.13, Cor. 3.14]{mo03}, \cite[Thm 4.4]{mi06} show that $\pi_d(H_d(q^2))$ and $\rho_d(U_q(\gl(m,n)))$ are the centralizers of each other in $\End(V^{\otimes d})$, thus
\[
\End_{\rho_d(U_q(\gl(m,n)))}(V^{\otimes d})=\pi_d(H_d(q^2)),
\]
and
\[
\End_{\pi_d(H_d(q^2))}(V^{\otimes d})=\rho_d(U_q(\gl(m,n)))
\]
by the double centralizer theorem of \cite[Thm 3.54]{cr81}. Moreover, \cite[Thm 5.16]{mo03}, \cite[Thm 5.1]{mi06} show that as an $H_d(q^2)\times U_q(\gl(m,n))$-bimodule,
\[
V^{\otimes d}=\oplus_{\lambda\in\Gamma(m,n;d)}H_\lambda\otimes V(\lambda),
\]
where $\Gamma(m,n;d)=\{\lambda=(\lambda_1,\lambda_2,\dots)\in\Par(d);\,\lambda_j\le n$ if $j>m\}$, $V(\lambda)$ is an irreducible representation of $U_q(\gl(m,n))$ indexed by $\lambda$ with $V(\lambda)\not\simeq V(\mu)$ if $\lambda\not=\mu$, and $H_\lambda$ is an irreducible representation of $H_d(q^2)$ indexed by $\lambda$.

In the ordinary (nonsuper) quantum case, this result is due to Jimbo \cite{j86}. In the super, yet non quantum, case, this is due to Berele-Regev \cite[Thm 3.20]{br87}. In the non super, non quantum, case, this is the original result of Schur \cite{sch27}, as refined by Weyl \cite{w53}. Proposition \ref{P10.2} is simply an extension of Jimbo's result, which is the case $n=0$, to the super ($n\ge 1$) case.

\section{Affine Schur-Weyl duality}\label{11}
Let us rephrase the Weyl-Schur duality of \cite[Theorem 5.16]{mo03} (and \cite[Theorem 5.1]{mi06}, \cite[Theorem 3.16]{zy}) in the context of quantum Lie superalgebras in a form useful for our generalization to the affine quantum super case.

\begin{prop}\label{P11.1} 
Fix integers $d$, $m$, $n\ge 2$. There is a unique left $H_d(q^2)$-module structure on $V^{\otimes d}$ such that $\wh T_i$ acts as $\check R_i$ $(1\le i<d);$ the action of $H_d(q^2)$ commutes with the natural action of $U_q^\sigma(\sl(m,n))$ on $V^{\otimes d}$. If $M$ is a right $H_d(q^2)$-module, define $J(M)=M\otimes_{H_d(q^2)}V^{\otimes d}$, with the natural $(\rho_d=\rho^{\otimes d}(\Delta^{(d-1)}))$ left $U_q^\sigma(\sl(m,n))$-module structure obtained from that on $V^{\otimes d}$. If $d<(n+1)(m+1)$ then the functor $M\mapsto J(M)$ is an equivalence from the category of finite dimensional $H_d(q^2)$-modules to the category of finite dimensional $U_q^\sigma(\sl(m,n))$-modules whose irreducible constituents all occur as constituents of $V^{\otimes d}$.
\end{prop}

Our main result is the next construction of a functor $\FF$, an equivalence of categories. The work is to check that the following extension to the affine context holds. But first we recall the definition of the fundamental representation $(\rho,V)$ of $U_{q,\AI}^\sigma=U_{q,\AI}^\sigma(\EE,\Pi,p)$. The space $V=V_{\0}\oplus V_{\1}$ is a superspace, thus $\Z/2$-graded, $V_{\0}=\oplus_{1\le i\le m}\C\ove_i$, $V_{\1}=\oplus_{m<i\le n'}\C\ove_i$, and there is a parity function $p:V\onto\Z/2$ with $p(\ove_i)$ being 0 on $V_{\0}$ and 1 on $V_{\1}$. The $\sigma$ acts as $\rho(\sigma)\ove_i=(-1)^{p(\ove_i)}\ove_i$ $(i\in I=\{1,\dots,n'=n+m\})$, thus $\rho(\sigma)=\diag(I_m,-I_n)$ in the basis $\{\ove_i\}$. Also $\rho(K_\gamma)\ove_i=q^{(\gamma,\ove_i)}\ove_i$ $(\gamma\in\Gamma\subset\EE)$,
\[
\rho(E_i)\ove_j=\delta(j,i+1)q_{i+1}^{-1}\ove_i,\quad\rho(F_i)\ove_j=\delta(i,j)q_{i+1}\ove_{i+1},
\] 
where we put $\ove_0=0=\ove_{n'+1}$. Thus in the basis $\{\ove_i;\,i\in I\}$ of $V$, 
\[
\rho(E_i)=q_{i+1}^{-1}E_{i,i+1},\quad \rho(F_i)=q_{i+1}E_{i+1,i},\quad \rho(K_\gamma)=\diag(q^{(\gamma,\ove_i)}).
\]
Then $\rho([E_{n'{}'},F_{n'{}'}])=\rho\left(\frac{K_{n'{}'}-K_{n'{}'}^{-1}}{q^{-1}-q}\right)$, confirming $(\QS3)$.

Recall that $\alpha_i=\ove_i-\ove_{i+1}$ ($1\le i<n'$) and $\alpha_0=-\ove_1+\ove_{n'}$, and that $(\ove_i,\ove_j)=\delta_{ij}(-1)^{p(\ove_i)}$. In particular $\rho(K_{\alpha_i})=\diag(I,q,q^{-1},I)$, $I$ signifies the identity matrix of the suitable size, $q$ at the $i$th place, if $1\le i<m$; $\rho(K_m)=\diag(I,q,q,I)$, $q$ at the $m$th and $(m+1)$st places; $\rho(K_{\alpha_i})= \diag(I,q^{-1},q,I)$, $q^{-1}$ at the $i$th place, if $m<i<n'$. Put $K_{\prod}=\prod_{1\le i\le n'{}'}K_{\alpha_i}$. Then $\rho(K_{\prod})=\diag(q,I,q)$. Put $E_{\prod}=E_1E_2\cdots E_{n'{}'}$ and $F_{\prod}=F_{n'{}'}\cdots F_2F_1$. Then $\rho(E_{\prod})=E_{1,n'}$ and $\rho(F_{\prod})=E_{n',1}$. Put $K_0=K_{\alpha_0}=K_{\prod}^{-1}$. Then $K_{\alpha_0}K_{\alpha_1}\dots K_{\alpha_{n'{}'}}$ -- as a central element of $U_q^\sigma$ -- acts as the identity.

For each $a\in\C^\times$ extend the representation $(\rho,V)$ of $U_{q,\AI}^\sigma$ to a $\wt U_q^\sigma=\wt U_q^\sigma(\EE,\Pi,p)$-module $(\rho,V(a))$ by $E_0=aF_{\Pi}$ and $F_0=a^{-1}E_{\Pi}$, thus $\rho(E_0)=a\rho(F_{\prod})$ and $\rho(F_0)=a^{-1}\rho(E_{\prod})$. Then $E_0F_0+F_0E_0=\frac{K_{\alpha_0}-K_{\alpha_0}^{-1}}{q^{-1}-q}$, which is compatible with $(\QS3)$ with the choice $q_0=q^{-1}$, thus $d_0=-1$.

Recall also that $\wt U_q^\sigma$ has a Hopf algebra structure $(\Delta,S,\varep)$ with $\Delta(\sigma)=\sigma\otimes\sigma$, $\Delta(K_\gamma)=K_\gamma\otimes K_\gamma$, $\Delta(F_i)=F_i\otimes K_{\alpha_i}^{-1}+\sigma^{p(\alpha_i)}\otimes F_i$, $\Delta(E_i)=E_i\otimes 1+ K_{\alpha_i}\sigma^{p(\alpha_i)}\otimes E_i$ where the parity $p$ of $\alpha_i$ is 0 except when $i=0$ or $i=m$ when it is 1. We defined $\Delta^{(k)}:\wt U_q^\sigma\to (\wt U_q^\sigma)^{\otimes (k+1)}$ to be $(\Delta\otimes 1^{\otimes (k-1)})\Delta^{(k-1)}$, where $\Delta^{(1)}=\Delta$. Also we defined $\rho_d:\wt U_q^\sigma\to\End(V^{\otimes d})$ by $\rho_d(x)=\rho^{\otimes d}\circ\Delta^{(d-1)}(x)$. Explicitly: $\rho_d(\sigma)=\rho(\sigma)^{\otimes d}$, $\rho_d(K_\gamma)=\rho(K_\gamma)^{\otimes d}$, and
\[
\rho_d(E_i)=\sum_{1\le k\le d}\rho(\sigma^{p(\alpha_i)}K_{\alpha_i})^{\otimes (k-1)}\otimes\rho(E_i)\otimes 1^{\otimes (d-k)},
\]
\[
\rho_d(F_i)=\sum_{1\le k\le d}\rho(\sigma^{p(\alpha_i)})^{\otimes (k-1)}\otimes\rho(F_i)\otimes \rho(K_{\alpha_i}^{-1})^{\otimes (d-k)},
\]
(recall that $p(\alpha_i)=0$ if $i\not=0$, $m$; $p(\alpha_0)=1=p(\alpha_m)$) since we have, by induction, 
\[
\Delta^{(d-1)}(E_i)=\sum_{1\le j\le d}(\sigma^{p(\alpha_i)}K_{\alpha_i})^{\otimes (j-1)}\otimes E_i\otimes 1^{\otimes (d-j)}=\Delta^{(d-2)}(E_i)\otimes 1+(\sigma^{p(\alpha_i)}K_{\alpha_i})^{\otimes (d-1)}\otimes E_i,
\]
\[
\Delta^{(d-1)}(F_i)=\sum_{1\le j\le d}(\sigma^{p(\alpha_i)})^{\otimes (j-1)}\otimes F_i \otimes (K_{\alpha_i}^{-1})^{\otimes (d-j)}=\Delta^{(d-2)}(F_i)\otimes K_{\alpha_i}^{-1}+(\sigma^{p(\alpha_i)})^{\otimes (d-1)}\otimes F_i.
\]

\begin{thm}\label{T11.2} 
Fix integers $d$, $m$, $n\ge 2$. There is a functor $\FF$ from the category of finite dimensional right $H_d^a(q^2)$-modules to the category of finite dimensional semisimple left $U_{q,\AI}^\sigma(\EE,\Pi,p)$-modules whose irreducible constituents are all submodules of $V^{\otimes d}$, defined as follows. Let $M$ be a right $H_d^a(q^2)$-module. Define $\FF(M)$ to be $J(M)$ as a $U_q^\sigma(\sl(m,n))$-module. Let the remaining generators of $U_{q,\AI}^\sigma(\EE,\Pi,p)$ act by
\[
(\rho_d(E_0))(m\otimes v)=\sum_{1\le j\le d}my_j^{-1}\otimes\rho^{\otimes d}(Y_{jE}^{(d)})v,\qquad Y_{jE}^{(d)}=(\sigma K_{\prod}^{-1})^{\otimes (j-1)}\otimes F_{\prod} \otimes 1^{\otimes (d-j)},
\]
\[
(\rho_d(F_0))(m\otimes v)=\sum_{1\le j\le d}my_j\otimes\rho^{\otimes d}(Y_{jF}^{(d)})v,\qquad Y_{jF}^{(d)}=\sigma^{\otimes (j-1)}\otimes E_{\prod} \otimes K_{\prod}^{\otimes (d-j)},
\]
for all $m\in M$ and $v\in V^{\otimes d}$, and $\rho_d(K_{\alpha_0})(m\otimes v)=m\otimes\rho(K_{\Pi}^{-1})^dv$, and
$\rho_d\left(\left[\frac{K_0-K_0^{-1}}{q-q^{-1}}\right]\right)(m\otimes v)$
\[
=\sum_{1\le j\le d}m\otimes\left((\rho(K_{\prod}^{-1})^{\otimes (j-1)}\otimes\rho\left(\frac{K_{\prod}-K_{\prod}^{-1}}{q-q^{-1}}\right)\otimes\rho(K_{\prod})^{\otimes (d-j)})\right)v.
\]
If $d<n'$ then the functor $M\mapsto \FF(M)$ is an equivalence from the category of finite dimensional $H_d^a(q^2)$-modules to the category of finite dimensional $U_{q,\AI}^\sigma(\sl(m,n))$-modules whose irreducible constituents all occur as constituents of $V^{\otimes d}$.
\end{thm}
This Theorem holds also for $d=1$. Its proof uses implicitly this case. We showed that our functor is an equivalence only for $d<n'$. Perhaps this result extends to $d<(n+1)(m+1)$ instead of $d<n'=m+n$. But our method of proof, which follows \cite{cp96}, requires $d<n'$. Note that $m\in M$ is unrelated to the integer $m=\dim V_{\0}$. 

\section{Operators are well-defined}\label{12}
The first task in order to prove the theorem is to show that the operators $\rho_d(E_0)$ and $\rho_d(F_0)$ are well defined. Then we need to check they satisfy the relations which define $U_{q,\AI}^\sigma$. We need to check only the new relations, those involving the generators $E_0$, $F_0$, $\left[\frac{K_0-K_0^{-1}}{q-q^{-1}}\right]$. We leave the verification of $(\QS2)$ for $E_0$, $F_0$, $K_{\alpha_0}$ to the reader. Then we need establish the equivalence of categories. In this section we check the operators are well defined. Thus we need to verify that
\[
(\rho_d(F_0))(m\wh T_i\otimes v)=(\rho_d(F_0))(m\otimes\wh T_iv)
\]
for all $m\in M$ and $v\in V^{\otimes d}$, namely as operators on $J(M)=M\otimes_{H_d(q^2)}V^{\otimes d}$ we have
\[
\sum_{1\le j\le d}\wh T_iy_j\otimes\rho^{\otimes d}(Y_{jF}^{(d)})=\sum_{1\le j\le d}y_j\otimes\rho^{\otimes d}(Y_{jF}^{(d)})\wh T_i.
\]
Recall that $\wh T_i^2-(q-q^{-1})\wh T_i-1=0$, $(\wh T_i-q)(\wh T_i+q^{-1})=0$, $\wh T_i-(q-q^{-1})=\wh T_i^{-1}$, $\wh T_iy_i\wh T_i=y_{i+1}$, and so $\wh T_iy_{i+1}=\wh T_i^2y_i\wh T_i=((q-q^{-1})\wh T_i+1)y_i\wh T_i=(q-q^{-1})y_{i+1}+y_i\wh T_i$.

If $j\not=i$, $i+1$, then $\wh T_i$ commutes with $y_j$ and with $\rho^{\otimes d}(Y_{jF}^{(d)})$. So it remains to show:
\[
\wh T_iy_i\otimes Y_{iF}^{(d)}+\wh T_iy_{i+1}\otimes Y_{i+1,F}^{(d)}=y_i\otimes Y_{iF}^{(d)}\wh T_i+y_{i+1}\otimes Y_{i+1,F}^{(d)}\wh T_i.
\]
Using the relations $\wh T_i$ satisfies, we see that the left side equals
\[
y_{i+1}\otimes \wh T_i^{-1} Y_{iF}^{(d)}+(q-q^{-1})y_{i+1}\otimes Y_{i+1,F}^{(d)}+y_i\wh T_i\otimes Y_{i+1,F}^{(d)}.
\]
Comparing to the right side we obtain
\[
y_i\otimes [\wh T_iY_{i+1,F}^{(d)}-Y_{i,F}^{(d)}\wh T_i]+y_{i+1}\otimes[\wh T_i^{-1}Y_{i,F}^{(d)}-Y_{i+1,F}^{(d)}\wh T_i^{-1}]=0.
\]
Hence it suffices to show: $\wh T_iY_{i+1,F}^{(d)}=Y_{i,F}^{(d)}\wh T_i$. Only two factors in the tensor product are affected, so we need only check that $\check R(\rho(\sigma)\otimes\rho(E_{\prod}))=(\rho(E_{\prod})\otimes\rho(K_{\prod}))\check R$.  Recall the explicit expression for $\check R$:
\[
\check R=\sum_{1\le i\le n'}(-1)^{p(\ove_i)}q^{(-1)^{p(\ove_i)}}E_{i,i}\otimes E_{i,i}
+\sum_{1\le i\not=j\le n'}(-1)^{p(\ove_i)}E_{j,i}\otimes E_{i,j}+\sum_{1\le i<j\le n'}(q-q^{-1})E_{i,i}\otimes E_{j,j}.
\]
and that $\rho(K_{\prod})=\diag(q,I,q)$, $\rho(E_{\prod})=E_{1n'}$ and $\rho(\sigma)=\diag(I_m,-I_n)$. The left side becomes
\[
\check R(\diag(I_m,-I_n)\otimes E_{1,n'})=qE_{11}\otimes E_{11}E_{1n'}+\sum_{i\not=j=1}(-1)^{p(\ove_i)}E_{1i}\otimes E_{i1}E_{1n'}
\]
and the right
\[
(E_{1n'}\otimes\diag(q,I,q))\check R=-q\cdot q^{-1}E_{1n'}E_{n'n'}\otimes E_{n'n'}+\sum_{i\not=j=n'}(-1)^{p(\ove_i)}E_{1n'}E_{n'i}\otimes E_{in'}\cdot q^{\delta(i,1)}.
\]
All terms in the sums are equal to one another, except that indexed by $i=n'$ in the first sum and that indexed by $i=1$ in the 2nd sum. The remaining two terms are $qE_{11}\otimes E_{1n'}+(-1)E_{1n'}\otimes E_{n'n'}$ in both cases, proving the required equality.

Similarly, to verify that $\rho_d(E_0)$ is well defined we need to show: 
\[
(\rho_d(E_0))(m\wh T_i\otimes v)=(\rho_d(E_0))(m\otimes\wh T_iv),
\]
i.e., that as operators on $J(M)=M\otimes_{H_d(q^2)}V^{\otimes d}$ we have
\[
\sum_{1\le j\le d}\wh T_iy_j^{-1}\otimes\rho^{\otimes d}(Y_{jE}^{(d)})=\sum_{1\le j\le d}y_j^{-1}\otimes\rho^{\otimes d}(Y_{jE}^{(d)})\wh T_i.
\]
If $j\not=i$, $i+1$, then $\wh T_i$ commutes with $y_j$ and with $\rho^{\otimes d}(Y_{jE}^{(d)})$. So it remains to show:
\[
\wh T_iy_i^{-1}\otimes Y_{iE}^{(d)}+\wh T_iy_{i+1}^{-1}\otimes Y_{i+1,E}^{(d)}=y_i^{-1}\otimes Y_{i+1,E}^{(d)}\wh T_i+y_{i+1}^{-1}Y_{i+1,E}^{(d)}\wh T_i.
\]
Using the relations $\wh T_i$ satisfies, we see this reduced to
\[
y_i^{-1}\otimes [Y_{i,E}^{(d)}(-\wh T_i+(q-q^{-1}))+\wh T_i^{-1}Y_{i+1,E}^{(d)}]+y_{i+1}^{-1} \otimes [-Y_{i+1,E}^{(d)}\wh T_i+\wh T_iY_{i,E}^{(d)}]=0.
\]
Hence it suffices to show: $Y_{i+1,E}^{(d)}\wh T_i=\wh T_iY_{i,E}^{(d)}$. Only two factors in the tensor product do not commute, so we are left with the need to show:
\[
(\rho(\sigma K_{\prod}^{-1})\otimes\rho(F_{\prod}))\check R=\check R(\rho(F_{\prod})\otimes 1),
\]
where $\rho(F_{\prod})=E_{n',1}$. The right side is
\[
\check R(E_{n',1}\otimes 1)=-q^{-1}E_{n',n'}E_{n',1}\otimes E_{n',n'}+\sum_{1\le j<n'=i}(-1)^{p(\ove_i)}E_{j,n'}E_{n',1}\otimes E_{n',j}(-1)^{1-p(\ove_j)},
\]
while the left side is $(\diag(I_m,-I_n)\diag(q^{-1},I,q^{-1})\otimes E_{n',1})\check R$
\[
=qq^{-1}E_{11}\otimes E_{n',1}E_{11}+\sum_{i=1<j\le n'}(-1)^{p(\ove_1)}E_{j,1}\otimes E_{n',1}E_{1,j}(-1)^{p(\ove_j)}q^{-\delta(j,n')}.
\]
All terms in the sums on both sides are equal except that indexed by $j=1$ on the right and $j=n'$ on the left, so the remaining two terms on both sides are equal to $E_{11}\otimes E_{n',1}-q^{-1}E_{n',1}\otimes E_{n',n'}$, proving the required equality.

\section{Relations $(\QS3)$, $(\QS4)(2)$}\label{13}
Consider the superbracket $[E_1,F_0]=E_1F_0-F_0E_1$. Then
\[
(\rho_d([E_1F_0-F_0E_1]))(m\otimes v)=\sum_{1\le j,k\le d}my_j\otimes\rho^{\otimes d}(X_{k,j})v
\]
where $X_{k,j}$ is $A_kB_j-B_jA_k$, we put $K_1$ for $K_{\alpha_1}$, and
\[
A_k=K_1^{\otimes (k-1)}\otimes E_1\otimes 1^{\otimes(d-k)},\quad B_j=\sigma^{\otimes (j-1)}\otimes E_{\prod}\otimes K_{\prod}^{\otimes (d-j)}.
\]
Recall that $\rho(E_1)=E_{12}$ and $\rho(E_{\prod})=E_{1,n'}$, $\rho(K_{\prod})=\diag(q,I,q)$, $\rho(K_1)=\diag(q,q^{-1},I)$, $\rho(\sigma)=\diag(I_m,-I_n)$. We apply $\rho^{\otimes d}$ but delete the $\rho$ from the notation for simplicity. Then $A_kB_j-B_jA_k$ is 0 if $j=k$ as $E_{12}E_{1,n'}=0=E_{1,n'}E_{12}$. It is easy to check that $A_k$ and $B_j$ commute when $j>k$. When $k>j$, all factors commute, except those at positions $k$ and $j$. At these two positions we get 
\[
K_1\otimes E_1\cdot E_{\prod}\otimes K_{\prod}-E_{\prod}\otimes K_{\prod}\cdot K_1\otimes E_1=qE_{1,n'}\otimes E_{12}-E_{1,n'}\otimes qE_{12}=0.
\]

Consider the superbracket $[E_0,F_0]=E_0F_0+F_0E_0$. Then
\[
(\rho_d([F_0,E_0]))(m\otimes v)=\sum_{1\le j,k\le d}my_j^{-1}y_k\otimes\rho^{\otimes d}(Y_{kE}^{(d)}\cdot Y_{jF}^{(d)}+Y_{jF}^{(d)}\cdot Y_{kE}^{(d)})v.
\]
Note that all factors in the tensor product in $Y_{kE}^{(d)}\cdot Y_{jF}^{(d)}+Y_{jF}^{(d)}\cdot Y_{kE}^{(d)}$ commute except those at the positions $j$, $k$. The terms corresponding to a pair $j<k$ add up to $\sigma K_{\prod}^{-1}\otimes F_{\prod}\cdot E_{\prod}\otimes K_{\prod}+E_{\prod}\otimes K_{\prod}\cdot \sigma K_{\prod}^{-1}\otimes F_{\prod}$. This equals $q^{-1}E_{1,n'}\otimes qE_{n',1}+(-1)q^{-1}E_{1,n'}\otimes qE_{n',1}=0$. If $k<j$ we get at the positions $(k,j)$ the sum $F_{\prod}\otimes 1\cdot \sigma\otimes E_{\prod}+\sigma\otimes E_{\prod}\cdot F_{\prod}\otimes 1=(F_{\prod}\sigma+\sigma F_{\prod})\otimes E_{\prod}$, and the dirst factor is 0.

When $j=k$ the term is
\[
(K_{\prod}^{-1})^{\otimes (j-1)}\otimes[E_{\prod} F_{\prod}+F_{\prod}E_{\prod}]\otimes
K_{\prod}^{\otimes (d-j)}.
\]
But 
\[
\rho(E_{\prod})\rho(F_{\prod})+\rho(F_{\prod})\rho(E_{\prod})=E_{1n'}E_{n'1}+ E_{n'1}E_{1n'}=\diag(1,0,\dots,0,1)=\frac{\rho(K_{\prod}-K_{\prod}^{-1})}{q-q^{-1}},
\]
and $(\QS3)$ follows.

To see that the relation $[F_0,F_0]=0$, namely $F_0^2=0$, is preserved by $\rho_d$, we consider 
\[
\rho_d(F_0)^2(m\otimes v)=\sum_{j,k}my_ky_j\otimes\rho^{\otimes d}(Y_{jF}^{(d)}\cdot Y_{kF}^{(d)})v.
\] 
It suffices to look at the factors in the tensor product where $E_{\prod}$ occur, as the other factors commute. Applying $\rho^{\otimes 2}$ to $E_{\prod}\otimes K_{\prod}\cdot\sigma\otimes E_{\prod}+\sigma\otimes E_{\prod}\cdot E_{\prod}\otimes K_{\prod}$ we get 
\[
E_{1n'}\diag(I,-I)\otimes\diag(q,I,q)E_{1n'}+\diag(I,-I)E_{1n'}\otimes E_{1n'}\diag(q,I,q),
\]
which is 0, when $j\not=k$. When $j=k$ we have $\rho(E_{\prod})^2=E_{1n'}^2=0$.

\section{Relations $(\QS4)(3)$}\label{14}
Next we verify that the relation relation $(\QS4)(3)$: $\br E_i,\br E_i,E_{i\pm 1}\ebr\ebr=0,$ $0\not=i\not=m,$ is preserved by $\rho_d$. This has to be verified only when one of the indices is 0. The two relations are $\br E_1,\br E_1,E_0\ebr\ebr=0$ and $\br E_{n'{}'},\br E_{n'{}'},E_0\ebr\ebr =0$. By the definition of $\br.,.\ebr$, since $-(\alpha_1,\alpha_0)=-(\ove_1-\ove_2,\ove_{n'}-\ove_1)=(\ove_1,\ove_1)=1$ and 
\[
-(\alpha_1,\alpha_1+\alpha_0)=-(\ove_1-\ove_2,\ove_1-\ove_2+\ove_{n'}-\ove_1)=-(\ove_2,\ove_2)=-1,
\] 
the first relation becomes
\[
0=[E_1,[E_1,E_0]_q]_{q^{-1}}=E_1(E_1E_0-qE_0E_1)-q^{-1}(E_1E_0-qE_0E_1)E_1\]
\[
=E_1^2E_0-(q+q^{-1})E_1E_0E_1+E_0E_1^2.
\]
Then to show vanishing of
\[
(\rho_d([E_1,[E_1,E_0]_q]_{q^{-1}}))(m\otimes v)=\sum_jmy_j^{-1}\otimes[\rho_d(E_1),[\rho_d(E_1),\rho^{\otimes d}(Y_{jE}^{(d)})]_q]_{q^{-1}}v,
\]
it suffices to show the vanishing of $\rho^{\otimes d}$ of $[\Delta^{(d-1)}(E_1),[\Delta^{(d-1)}(E_1),Y_{j,E}^{(d)}]_q,]_{q^{-1}}$. When $d=1$ this leads to $\rho([E_1,[E_1,E_0]_q]_{q^{-1}})=[E_{12},[E_{12},E_{n',1}]_q]_{q^{-1}}$, which is 0 since $E_{12}^2=0$, $E_{12}E_{n',1}=0$. When $d=2$ we are led to 
\[
[\Delta(E_1),[\Delta(E_1),F_{\prod}\otimes 1+\sigma K_{\prod}^{-1}\otimes F_{\prod}]_q]_{q^{-1}}\]
\[
=[E_1\otimes 1+K_1\otimes E_1,[E_1\otimes 1+K_1\otimes E_1,\sigma K_{\prod}^{-1}\otimes F_{\prod}]_q]_{q^{-1}}
\]
\[
+[E_1\otimes 1+K_1\otimes E_1,[E_1\otimes 1+K_1\otimes E_1,F_{\prod}\otimes 1]_q]_{q^{-1}}.
\]
Apply $\rho^{\otimes 2}$. The $[.,.]_q$ of the first summand is
\[
E_{12}\otimes E_{n',1}-q(q^{-1} E_{12}\otimes E_{n',1}+\sigma\diag(1,q^{-1},I,q^{-1})\otimes E_{n',2})=-q\sigma\diag(1,q^{-1},I,q^{-1})\otimes E_{n',2},
\]
so the $[.,.]_{q^{-1}}$ is $-E_{12}\otimes E_{n',2}+E_{12}\otimes E_{n',2}=0$. The $[.,.]_q$ of the second summand is
\[
E_{n',1}\otimes E_{12}-q(E_{n',2}\otimes 1+E_{n',1}\otimes E_{12})=(1-q^2)E_{n',1}\otimes E_{12}-qE_{n',2}\otimes 1,
\]
so the $[.,.]_{q^{-1}}$ is $-qE_{n',2}\otimes E_{12}-q^{-1}((1-q^2)E_{n',2}\otimes E_{12}-q\cdot q^{-1}E_{n',2}\otimes E_{12})=0$.

In general, we need to verify that after applying $\rho^{\otimes d}$, that we shall omit to simplify the notation, the sum $\sum_{1\le s,t\le d}a(s,t,j)$ is mapped to zero for each $j$, where
\[
a(s,t,j)=[K_1^{\otimes (s-1)}\otimes E_1\otimes 1^{\otimes (d-s)},
[K_1^{\otimes (t-1)}\otimes E_1\otimes 1^{\otimes (d-t)},(\sigma K_{\prod}^{-1})^{\otimes (j-1)}\otimes E_{n',1}\otimes 1^{\otimes (d-j)}]_q]_{q^{-1}}.
\]
Fix $j$. The term $s=t=j$ is zero since this case reduces to that of $d=1$, as the components at all other positions commute. So ($\rho^{\otimes 3}$ of) $a(j,j,j)=0$.

Fix $j'\not=j$. If $s,\,t$ range over the set $\{j,\,j'\}$, this reduces to the case of $d=2$, for the same reason. In particular, the sum of the terms $a(j',j,j)$, $a(j,j',j)$, $a(j',j',j)$ is zero.

Fix $\{j'{}',\,j'\}$, $j\not=j'\not=j'{}'\not=j$. It remains to show that $a(j',j'{}',j)+a(j'{}',j',j)=0$ for all triples $\{j,\,j'{}',\,j'\}$. As the components in the other positions commute, it suffices to consider the case where $d=3$. There are 3 cases: $j=1$, 2, 3. Consider $j=1$. We have
\[
[K_1\otimes E_1\otimes 1,[K_1\otimes K_1\otimes E_1,E_{n',1}\otimes 1\otimes 1]_q]_{q^{-1}}\qquad (s=2,\quad t=3).
\]
We first compute the inner bracket $[.,.]_q$ using $K_1E_{n',1}-qE_{n',1}K_1=(1-q^2)E_{n',1}$. Then using $E_1K_1=q^{-1}E_1$, $K_1E_1=qE_1$, this term is seen to be
$(1-q^2)(q^{-1}-q)E_{n',1}\otimes E_1\otimes E_1$. The term corresponding to $s=3$, $t=2$ is
\[
[K_1\otimes K_1\otimes E_1,[K_1\otimes E_1\otimes 1,E_{n',1}\otimes 1\otimes 1]_q]_{q^{-1}}=(1-q^2)(q-q^{-1})E_{n',1}\otimes E_1\otimes E_1
\]
by similar computations, so the sum of these two terms is zero. When $j=2$, for $s=1$, $t=3$ we have
\[
[E_1\otimes 1\otimes 1,[K_1\otimes K_1\otimes E_1,\sigma K_{\prod}^{-1}\otimes E_{n',1}\otimes 1]_q]_{q^{-1}}
\]
\[
=(1-q^2)(E_1K_1\sigma K_{\prod}^{-1}-q^{-1}K_1\sigma K_{\prod}^{-1}E_1)\otimes E_{n' 1}\otimes E_1=0,
\]
and for $s=3$, $t=1$
\[
[K_1\otimes K_1\otimes E_1,[E_1\otimes 1\otimes 1,\sigma K_{\prod}^{-1}\otimes E_{n',1}\otimes 1]_q]_{q^{-1}}
\]
is zero since the first component in the inner $[.,.]_q$ is $E_1\sigma K_{\prod}^{-1}-q\sigma K_{\prod}^{-1} E_1=E_1-q\cdot q^{-1}E_1=0$. Finally, when $j=3$, 
\[
[E_1\otimes 1\otimes 1,[K_1\otimes E_1\otimes 1,\sigma K_{\prod}^{-1}\otimes\sigma K_{\prod}^{-1}\otimes E_{n',1}]_q]_{q^{-1}}
\]
is 0 since the 3rd component at the inner $[.,.]_q$ is $E_1\sigma K_{\prod}^{-1}-q\sigma K_{\prod}^{-1} E_1=0$, and
\[
[K_1\otimes E_1\otimes 1,[E_1\otimes 1\otimes 1,\sigma K_{\prod}^{-1}\otimes\sigma K_{\prod}^{-1}\otimes E_{n',1}]_q]_{q^{-1}}=0
\]
since the first component in the inner $[.,.]_q$ is again $E_1\sigma K_{\prod}^{-1}-q\sigma K_{\prod}^{-1} E_1=0$.\\

We also need to check that the relation 
\[
0=\br E_{n'{}'},\br E_{n'{}'},E_0\ebr\ebr=[E_{n'{}'},[E_{n'{}'},E_0]_q]_{q^{-1}}
\]
(second equality from $-(\alpha_{n'{}'},\alpha_0)=-(\ove_{n'{}'}-\ove_{n'},-\ove_1+\ove_{n'})=1$, $-(\alpha_{n'{}'},\alpha_0+\alpha_{n'{}'})=-1$)
is preserved by $\rho_d$, namely that so is $[\Delta^{(d-1)}(E_{n'{}'}),[\Delta^{(d-1)}(E_{n'{}'}),E_0]_q]_{q^{-1}}=0$. Recall that
\[
\Delta^{(d-1)}(E_{n'{}'})=\sum_{0\le j\le d-1}(\sigma^{p(\alpha_{n'{}'})}K_{\alpha_{n'{}'}})^{\otimes j}\otimes E_{n'{}'}\otimes 1^{\otimes (d-1-j)}.
\]
Recall that $p(\alpha_{n'{}'})=p(\alpha_i)=0$ if $i\not=0,\,m$, and $p(\alpha_0) = p(\alpha_m) =1$. The verification of this case is similar to that of the previous case, and is left to the reader.

This completes the verification that the relations $(\QS4)(3)$ are preserved under $\rho_d$. 

The relations $(\QS5)(3)$, in which the $E$ are replaced by $F$, are verified by analogous computations.

\section{Relations $(\QS4)(4')$}\label{15}
Finally we need to theck that the relation $(\QS4)(4')$, which is $[\br\br E_{n'{}'},E_0 \ebr,E_1\ebr,E_0]=0$, equivalently $[\br\br E_1,E_0\ebr,E_{n'{}'}\ebr,E_0]=0$, is preserved under $\rho_d$. Consider the last relation. Since 
\[
-(\alpha_1,\alpha_0)=-(\ove_1-\ove_2,\ove_{n'{}'}-\ove_1)=(\ove_1,\ove_1)=1,\quad \br E_1,E_0\ebr=[E_1,E_0]_q=E_1E_0-qE_0E_1.
\] 
Since 
\[
-(\alpha_1+\alpha_0,\alpha_{n'{}'})=-(\ove_1-\ove_2+\ove_{n'}-\ove_1,\ove_{n'{}'}-\ove_{n'})=(\ove_{n'},\ove_{n'})=-1,
\]
the $\br\br E_1,E_0\ebr,E_{n'{}'}\ebr$ is $[[E_1,E_0]_q,E_{n'{}'}]_{q^{-1}}$, and the remaining bracket, $[\ast,E_0]$, is $\ast E_0+E_0\ast$, since $p(E_1E_0E_{n'{}'})=1$ and $p(E_0)=1$. We need to show then
\[
[[[\rho_d(E_1),\rho_d(E_0)]_q,\rho_d(E_{n'{}'})]_{q^{-1}},\rho_d(E_0)]=0.
\]
As
\[
\rho_d(E_1)=\rho^{\otimes d}(\Delta^{(d-1)}(E_1))=\sum_{0\le j<d} K_{\alpha_1}^{\otimes j}\otimes\rho(E_1)\otimes 1^{\otimes (d-1-j)}
\]
and
\[
(\rho_d(E_0))(m\otimes v)=\sum_{1\le j\le d}my_j^{-1}\otimes\rho^{\otimes d}(Y_{jE}^{(d)})v,
\quad Y_{jE}^{(d)}=(\sigma K_{\prod}^{-1})^{\otimes (j-1)}\otimes F_{\prod}\otimes 1^{\otimes (d-j)},
\]
we need consider the sum of terms of the form (as before, to simplify the notation, by $E_1$, $K_1=K_{\alpha_1}$, $F_{\prod}$, $K_{\prod}$, $\sigma$, $E_{n'{}'}$, $K_{n'{}'}$ we mean below their images under $\rho$: $E_{12}$, $\diag(q,q^{-1},I)$, $E_{n',1}$, $\diag(q,I,q)$, $\diag(I_m,-I_n)$, $E_{n'{}',n'}$, $\diag(I,q^{-1},q)$)
\[
a(j_1,j_2,j_3,j_4)=[[[K_1^{\otimes (j_1-1)}\otimes E_1\otimes 1^{\otimes (d-j_1)}, (\sigma K_{\prod}^{-1})^{\otimes (j_2-1)}\otimes F_{\prod}\otimes 1^{\otimes (d-j_2)}]_q,
\]
\[
(\sigma K_{n'{}'})^{\otimes (j_3-1)}\otimes E_{n'{}'}\otimes 1^{\otimes (d-j_3)}]_{q^{-1}},(\sigma K_{\prod}^{-1})^{\otimes (j_4-1)}\otimes F_{\prod}\otimes 1^{\otimes (d-j_4)}].
\]

To keep track of the accounting, the procedure will be to fix $(j_2,j_4)$, and consider the sum of the terms $a$ for all the possibilities for $j_1$, $j_3$. In all cases the sum is zero. There are too many cases to record all computations here, but the technique is as in the previous section. We describe a few cases. If all $j_i$ are equal to the same $j$, then we may assume $d=1$, as the other components in the tensor product commute. In this case we are reduced to the computation (recall that we apply $\rho$ although this is omitted from the notation):
\[
[[[E_1,F_{\prod}]_q,E_{n'{}'}]_{q^{-1}},F_{\prod}].
\]
The inner bracket is $[E_{12},E_{n',1}]_q=-qE_{n',2}$. The bracket of this with $E_{n'{}'}=E_{n'{}',n'}$ is $E_{n'{}',2}$, and this bracketed with $F_{\prod}=E_{n',1}$ is 0 as $E_{n'{}',2}E_{n',1}=0=E_{n',1}E_{n'{}',2}$. 

Next we consider the case of $j_2=j_4=j$, and $j_1$, $j_3$ in $\{j,\,j'\}$. We may work with $d=2$, so $j=1$ or 2. When $j=1$, $j_1=1$, $j_3=2$, we get
\[
[[[E_{12}\otimes 1,E_{n',1}\otimes 1]_q,\sigma K_{n'{}'}\otimes E_{n'{}'}]_{q^{-1}},E_{n',1}\otimes 1]=q(1-q^{-2})[E_{n',2}\otimes E_{n'{}'},E_{n',1}\otimes 1]=0.
\]
When 
$j=1$, $j_1=2$, $j_3=1$, we get
\[
[[[K_1\otimes E_1,F_{\prod}\otimes 1]_q,E_{n'{}'}\otimes 1]_{q^{-1}},F_{\prod}\otimes 1]=-q^{-1}(1-q^2)[E_{n'{}',1}\otimes E_1,E_{n',1}\otimes 1]=0.
\]
And when $j_1=2=j_3$,
\[
[[[K_1\otimes E_1,F_{\prod}\otimes 1]_q,\sigma K_{n'{}'}\otimes E_{n'{}'}]_{q^{-1}},F_{\prod}\otimes 1]
\]
is zero since $[.,.]_q$ is $(1-q^2)E_{n',1}\otimes E_{12}$, and the bracket of $E_{12}$ with $E_{n'{}'}$ is zero.

When $j=2$, $j_1=1$, $j_3=2$:
\[
[[[E_1\otimes 1,\sigma K_{\prod}^{-1}\otimes F_{\prod}]_q,\sigma K_{n'{}'}\otimes E_{n'{}'}]_{q^{-1}},\sigma K_{\prod}^{-1}\otimes F_{\prod}]
\]
is zero since $[.,.]_q$ is $(E_1-q\cdot q^{-1}E_1)\otimes F_{\prod}=0$. When $j=2$, $j_1=2$, $j_3=1$:
\[
[[[K_1\otimes E_1,\sigma K_{\prod}^{-1}\otimes F_{\prod}]_q,E_{n'{}'}\otimes 1]_{q^{-1}},\sigma K_{\prod}^{-1}\otimes F_{\prod}]
\]
vanishes since $[.,.]_{q^{-1}}$ is $\ast\otimes E_{n',2}$, and $E_{n',2}$ times $F_{\prod}=E_{n',1}$ (on the right and on the left) is 0. The last case, where $j_1=1=j_3$ is zero as the $[.,.]_q$ is the same as in the case $j=2$, $j_1=1$, $j_3=2$.

If $j_2=j_4=j$ and $j_1$, $j_3\not=j$ then we may work with $d=3$. Thus if $j=1$, $(j_1,j_3)$ is (2,3) or (3,2). If $j=2$, $(j_1,j_3)$ is (1,3) or (3,1). If $j=3$, $(j_1,j_3)$ is (1,2) or (2,1).

If $j_2\not=j_4$, and $j_1$, $j_3\in\{j_2,j_4\}$, then $(j_2,j_4)=(1,2)$ and $(j_1,j_3)=(1,2)$ and $(2,1)$, or $(j_2,j_4)=(2,1)$ and $(j_1,j_3)=(1,2)$ and $(2,1)$. If $j_1$, $j_3\in\{j_2,j_4,j'\}$ but not both in $\{j_2,j_4\}$, then we can work with $d=3$. The pair $(j_1,j_3)$ is $(j_2,j')$, $(j_4,j')$, $(j',j_2)$, $(j',j_4)$, that is, one of $j_1$, $j_3$ is in $\{j_2,\,j_4\}$, the other is not.

When $j_2\not=j_4$, and $j_1$, $j_3\notin\{j_2,j_4\}$, then we work in $d=3$ if $j_1=j_3$ and with $d=4$ if not. In particular it suffices to work with $d\le 4$, and in each case the computation is reduced to a simple matrix multiplication, that can be verified by hand or by machine.

This computation verifies $(\QS4)(4')$. The verification of the cases of $(\QS5)$, where the generators $E$ are replaced by the generators $F$, is similar.

We conclude that the formulae for $\rho_d(E_0)$ and $\rho_d(F_0)$ then define a representation of $U_{q,\AI}^\sigma(\EE,\Pi,p,\Gamma)$. 

If $f:M\to M'$ is a homomorphism of $H_d^a(q^2)$-modules, define $\FF(f):\,\FF(M)\to\FF(M')$ by $(\FF(f))(m\otimes v)=f(m)\otimes v$. Then $\FF(f)$ is a well-defined homomorphism of $U_{q,\AI}^\sigma(\EE,\Pi,p)$-modules, so that $\FF$ is a functor between the categories of representations as specified in the theorem.

\section{The functor $\FF$ is an equivalence}\label{16}
Assume from now on that $d< n'$. To show that the functor $\FF$ -- which we have seen is a well-defined functor between the categories specified in the theorem -- is an equivalence, one has to show:\\
$(a)$ Every finite dimensional $U_{q,\AI}^\sigma(\EE,\Pi,p,\Gamma)$-module $W$ which is completely reducible and each of its irreducible constituents is a constituent of $V^{\otimes d}$ is isomorphic to $\FF(M)=M\otimes_{H_d(q^2)}V^{\otimes d}$ for some $H_d^a(q^2)$-module $M$.\\
$(b)$ $\FF$ is bijective on sets of morphisms.

To prove $(a)$, by Proposition \ref{P11.1} we assume that $W=J(M)$ for some $H_d(q^2)$-module $M$. We shall construct the action of the $y_j^{\pm 1}$ on $M$ from the given action of $\rho_d(E_0)$, $\rho_d(F_0)$, $\rho_d(\HH)$ on $W$.

\begin{lem}\label{L16.1} 
$(a)$ Let $M$ be a finite dimensional $H_d(q^2)$-module. Fix $v\in V^{\otimes d}$. Suppose that the projection of $v$ to each isotypical component of $J(M)$ is nonzero. Then the map $M\to J(M)$, $m\mapsto m\otimes v$, is injective.\\
$(b)$ Recall that $\{\ove_1,\dots,\ove_{n'}\}$ denotes the standard basis of $V$. Suppose $v=\ove_{i_1}\otimes\dots\otimes\ove_{i_d}\in V^{\otimes d}$, where $i_1,\dots,i_d\in\{1,\dots,n'\}$ are distinct. Then $V^{\otimes d}=U_{q,\AI}^\sigma(\EE,\Pi_0,p,\Gamma)\cdot v$, where $\Pi_0=\{\alpha_1,\dots,\alpha_{n'{}'}\}$. In particular $v$ satisfies the condition stated in $(a)$.
\end{lem}
\begin{proof}
As in \cite[Lemma 4.3]{cp96}, $(a)$ follows from Proposition \ref{P11.1}, and $(b)$ is clear.
\end{proof}
\begin{lem}\label{L16.2} 
$(a)$ For $j$ $(1\le j<n')$ put $a(j)=\ove_2\otimes\dots\otimes\ove_j$, $b(j)=\ove_{j+1}\otimes\dots\otimes\ove_d$,
\[
v^{(j)}=a(j)\otimes\ove_{n'}\otimes b(j),\qquad w^{(j)}=a(j)\otimes\ove_1\otimes b(j).
\]
Then there exists $\alpha_{jF}\in\End_{\C}M$ with 
\[
(\rho_d(F_0))(m\otimes v^{(j)})=\alpha_{jF}(m)\otimes\rho^{\otimes d}(Y_{jF}^{(d)})v^{(j)}
\]
and 
$\alpha_{jE}\in\End_{\C}M$ with 
\[
(\rho_d(E_0))(m\otimes w^{(j)})=\alpha_{jE}(m)\otimes\rho^{\otimes d}(Y_{jE}^{(d)})w^{(j)}.
\]
We have $\rho^{\otimes d}(Y_{jF}^{(d)})v^{(j)}=\pm w^{(j)}$, and $\rho^{\otimes d}(Y_{jE}^{(d)})w^{(j)}=\pm v^{(j)}$.
\end{lem}
\begin{proof}
For $\tau$ in the symmetric group $S_d$ on $d$ letters, put 
\[
w_\tau^{(j)}=(\ove_{\tau(2)}\otimes\dots\otimes\ove_{\tau(j)})\otimes\ove_{\tau(1)}\otimes(\ove_{\tau(j+1)}\otimes\dots\otimes\ove_{\tau(d)}). 
\]
The set $\{w_\tau^{(j)};\tau\in S_d\}$ spans the subspace of $V^{\otimes d}$ of weight $\lambda_d=\ove_1+\ove_2+\dots+\ove_d$. Indeed, $(\rho_d(K_\gamma))\ove_i=q^{(\gamma,\ove_i)}\ove_i$, so $(\rho_d(K_\gamma))w_\tau^{(j)}=q^{(\gamma,\ove_1+\dots+\ove_d)}w_\tau^{(j)}$. Note that $\rho_d(K_\gamma)\rho_d(F_0)=q^{-(\gamma,\alpha_0)}\rho_d(F_0)\rho_d(K_\gamma)$, hence $\rho_d(F_0)$ adds $\ove_1-\ove_{n'}$ to the weight, hence it takes $\ove_{n'}$ to $\ove_1$. Hence for every $m\in M$ we have
\[
(\rho_d(F_0))(m\otimes v^{(j)})=\sum_{\tau\in S_d}m_\tau\otimes w_\tau^{(j)}
\]
for some $m_\tau\in M$. By the definition of $\check R$, $w_\tau^{(j)}$ is a nonzero scalar multiple of $h\cdot w^{(j)}$ for some $h\in H_d(q^2)$, $h=h(\tau)$. Hence $(\rho_d(F_0))(m\otimes v^{(j)})$ equals $m'\otimes w^{(j)}$ for some $m'\in M$. Then there exists $\alpha_{jF}\in\End_\C M$ with $m'=\alpha_{jF}(m)$ for all $m\in M$ by Lemma \ref{L16.1}. The existence $\alpha_{jE}\in\End_\C M$ is proven analogously.
\end{proof}
\begin{lem}\label{L16.3} 
For all $m\in M$ and $v\in V^{\otimes d}$ we have
\[
(\rho_d(E_0))(m\otimes v)=\sum_{1\le j\le d}\alpha_{jE}(m)\otimes\rho^{\otimes d}(Y_{jE}^{(d)})v,\quad 
(\rho_d(F_0))(m\otimes v)=\sum_{1\le j\le d}\alpha_{jF}(m)\otimes\rho^{\otimes d}(Y_{jF}^{(d)})v.
\]
\end{lem}
\begin{proof}
Recall that $K_\gamma F_0K_\gamma^{-1}=q^{-(\gamma,\alpha_0)}F_0$, where $\alpha_0=\ove_{n'}-\ove_1$, and $\rho(K_\gamma)\ove_i=q^{(\gamma,\ove_i)}\ove_i$. Hence $\rho_d(K_\gamma)\rho_d(F_0)(m\otimes v)$, where $v=\ove_{i_1}\otimes\dots\otimes\ove_{i_d}$, is $q^{(\gamma,-\alpha_0+\ove_{i_1}+\dots+\ove_{i_d})}\rho_d(F_0)(m\otimes v)$, and this will be 0 if no $i_j$ is $n'$, as then $-\ove_{n'}+\ove_1+\ove_{i_1}+\dots+\ove_{i_d}$ cannot be a weight of $V^{\otimes d}$. So we may assume some component of $v$ is $\ove_{n'}$.

Let $r\ge 0$, $s\ge 1$, $r+s\le d$, $1\le j_1<j_2<\dots<j_r\le d$, $1\le j'_1<j'_2<\dots< j'_s\le d$, assume $\{j_1,\dots,j_r\}\cap\{j'_1,\dots,j'_s\}=\emptyset$. Write $j=(j_1,\dots,j_r)$, $j'=(j'_1,\dots,j'_s)$. Let $V^{(j,j')}$ be the subspace of $V^{\otimes d}$ spanned by the vectors which have $\ove_1$ in positions $j_1,\dots,j_r$; $\ove_{n'}$ in positions $j'_1,\dots,j'_s$; and vectors from $\{\ove_2,\dots,\ove_{n'{}'}\}$ in the remaining positions. We prove the lemma when $v$ is in $V^{(j,j')}$ for all $j$, $j'$ in two steps.\\
$(i)$ For $s=1$, by induction on $r$.\\
$(ii)$ For all $r$, by induction on $s$.\\
By Lemma \ref{L16.1}(b), applied to the subalgebra of $U_q^\sigma$ generated by the $E_i$, $F_i$, $K_{\alpha_i}^{\pm 1}$ for $i\in\{2,\dots,n'{}'-1\}$, to prove our lemma for all $v\in V^{(j,j')}$ it suffices to prove it for one $0\not= v\in V^{(j,j')}$ whose components have no vector from $\{\ove_2,\dots,\ove_{n'{}'}\}$ twice. Such vectors exist since $1\le d+1-r-s\le d\le n'{}'$.\\
{\em Proof of step $(i)$}. Here $s=1$. The case of $r=0$ follows from Lemma \ref{L16.2}(a): take 
\[ 
v=a(j'_1)\otimes\ove_{n'}\otimes b(j'_1), \qquad w=a(j'_1)\otimes\ove_1\otimes b(j'_1),
\]
(recall: $a(j)=\ove_2\otimes\dots\otimes\ove_j$, $b(j)=\ove_{j+1}\otimes\dots\otimes\ove_d$). As $Y_{jF}^{(d)}=\sigma^{\otimes (j-1)}\otimes E_{\prod}\otimes  K_{\prod}^{\otimes (d-j)}$, and $\rho(E_{\prod})=E_{1,n'}$, we have $\rho^{\otimes d}(Y_{j'_1,F}^{(d)})v=w$ times $(-1)^{\max(0,j'_1-m)}$, and $\rho^{\otimes d}(Y_{j,F}^{(d)})v=0$ for all $j\not=j'_1$, hence $(\rho_d(F_0))(m\otimes v)=\sum_{1\le j\le d}\alpha_{jF}(m)\otimes\rho^{\otimes d}(Y_{j,F}^{(d)})v$, where $\alpha_{jF}(m)=(-1)^{\max(0,j-m)}$. Recall that the integer $m=\dim V_{\0}$ in the exponent is not $m\in M$ on the left.

Assume Step $(i)$ holds for $r-1$. Put $\wt j=(j_2,\dots,j_r)$. Define $v'\in V^{(\wt j,j')}$ to be a pure tensor with $\ove_2$ in the $j_1$ position, and distinct vectors from $\{\ove_3,\dots,\ove_{n'{}'}\}$ in the remaining positions. Then $v=\rho_d(E_1)v'$. Indeed, recall that $\rho_d(E_1)=\sum_k \rho(K_{\alpha_1})^{\otimes (k-1)}\otimes\rho(E_1)\otimes 1^{\otimes (d-k)}$, that $\rho(E_1)\ove_j=\delta(2,j)\ove_1$, and that $v'$ has $\ove_2$ only at position $j_1$ (and $\ove_1$ only at positions $j_2,\dots,j_r$), so only $k=j_1$ survives in the sum over $k$ which defines $\rho_d(E_1)$, and $(\rho_d(E_1))v'=v$ as $\rho(K_1)=\diag(q,q^{-1},I)$ acts nontrivially only on $\ove_1$ and $\ove_2$.

Define $v'{}'$ by replacing $\ove_{n'}$ in position $j'=j'_1$ in $v'$ by $\ove_1$, and $v'{}'{}'$ by replacing $\ove_2$ in position $j_1$ in $v'{}'$ by $\ove_1$. Now $r(v')=r-1$, so we can apply the induction on $r$ (in the 3rd equality below, and $(\QS3)$ in the second).
\[
(\rho_d(F_0))(m\otimes v)=\rho_d(F_0)\rho_d(E_1)(m\otimes v')=\rho_d(E_1)\rho_d(F_0)(m\otimes v')
\]
\[
=\rho_d(E_1)\sum_{1\le \ell\le d}\alpha_{\ell,F}(m)\otimes\rho^{\otimes d}(Y_{\ell,F}^{(d)})v'.
\]
Recall again that $Y_{\ell,F}^{(d)}$ is $\sigma^{\otimes (\ell-1)}\otimes E_{\prod}\otimes  K_{\prod}^{\otimes (d-\ell)}$, and $\rho(E_{\prod})=E_{1,n'}$, and $\ove_{n'}$ occurs only at position $j'_1$ in $v'$. Then only $\ell=j'_1$ survives in the sum, which becomes a multiple of $v'{}'$, by a sign $\iota$, which is $-1$ if the number of factors of the form $\ove_a$ with $a>m$ in position less than $j'_1$ is odd. Since $\ove_2$ occurs in $v'{}'$ only in position $j_1$, in the sum defining $\rho_d(E_1)$ only the summand indexed by $k=j_1$ survives when acting on $v'{}'$, and it is $\rho(K_1)^{\otimes (j_r-1)}\otimes\rho(E_1)\otimes 1^{\otimes (d-j_r)}$. So $\rho_d(E_1)$ maps $v'{}'$ to $v'{}'{}'$. We obtain $\alpha_{j'_1,F}(m)$ times $\iota v'{}'{}'=\rho^{\otimes d}(Y_{j'_1,F}^{(d)})v$. For other $j$ we have $0=\rho^{\otimes d}(Y_{j,F}^{(d)})v$. So we end up with $\sum_j\alpha_{j,F}(m)\otimes\rho^{\otimes d}(Y_{j,F}^{(d)})v$, completing step $(i)$. \\

\n {\em Proof of step $(ii)$}. Assume the lemma holds for all $v\in V^{(j,j')}$ with less than $s$ components $\ove_{n'}$. As in Step $(i)$, it suffices to prove the claim for one element $v \not=0$ in $V^{(j,j')}$ which has distinct entries from $\{\ove_2,\dots,\ove_{n'{}'-1}\}$ in the remaining positions. Fix such a $v$. Let $v'$ be the tensor obtained from $v$ on replacing $\ove_{n'}$ in positions $j'_{s-1}$ and $j'_s$ by $\ove_{n'{}'}$. We claim that 
\[
\rho_d(F_{n'{}'})^2v'=(q+q^{-1})v. 
\]
To see this, recall that $\rho(F_{n'{}'})=E_{n',n'{}'}$, $p(\alpha_{n'{}'})=0$,
\[
\rho_d(F_{n'{}'})=\sum_{1\le k\le d}1^{\otimes (k-1)}\otimes\rho(F_{n'{}'})\otimes \rho(K_{\alpha_{n'{}'}}^{-1})^{\otimes (d-k)}, \qquad \rho(K_{n'{}'}^{-1})=\pmatrix I&& \\ &q& \\ && q^{-1}\endpmatrix .
\]
So in $\rho_d(F_{n'{}'})^2v'$ the sum over $k$ in each $\rho_d(F_{n'{}'})$ reduces to $k= j'_{s-1}$, $j'_s$, and all factors in positions $\not= j'_{s-1}$, $j'_s$ in each summand, commute. At these two positions the components of $v'$ are $\ove_{n'{}'}\otimes\ove_{n'{}'}$ and those of $\rho_d(F_{n'{}'})^2$ are
\[
(\rho(F_{n'{}'})\otimes\rho(K_{n'{}'}^{-1})+1\otimes\rho(F_{n'{}'}))
(\rho(F_{n'{}'})\otimes\rho(K_{n'{}'}^{-1})+1\otimes\rho(F_{n'{}'}))
\]
\[
=\rho(F_{n'{}'})\otimes\rho(K_{n'{}'}^{-1}F_{n'{}'})+\rho(F_{n'{}'})\otimes\rho(F_{n'{}'}K_{n'{}'}^{-1})
\]
as $\rho(F_{n'{}'})^2=0$. So $\rho_d(F_{n'{}'})^2v'$ equals
\[
1^{\otimes (j'_{s-1}-1)}\otimes\rho(F_{n'{}'})\otimes\rho(K_{n'{}'}^{-1})^{\otimes (j'_s-1-j'_{s-1})}\otimes(\rho(K_{n'{}'}^{-1}F_{n'{}'})+\rho(F_{n'{}'}K_{n'{}'}^{-1}))\otimes\rho(K_{n'{}'}^{-2})^{\otimes (d-j'_s)}v'.
\]
Now $\rho(F_{n'{}'})\ove_{n'{}'}=\ove_{n'}$, $\rho(K_{n'{}'}^{-1}F_{n'{}'})\ove_{n'{}'}=q^{-1}\ove_{n'}$, $\rho(F_{n'{}'}K_{n'{}'}^{-1})\ove_{n'{}'} = q\ove_{n'}$, $\rho(K_{n'{}'}^{-1})^{\otimes (j'_s-1-j'_{s-1})}$ acts trivially, so in conclusion $v=\frac{1}{q+q^{-1}}\rho_d(F_{n'{}'})^2v'$, as claimed. 

To continue we use the equality $(\QS5)(3)$:
\[
\rho_d(F_0)\rho_d(F_{n'{}'})^2=(q+q^{-1})\rho_d(F_{n'{}'})\rho_d(F_0)\rho_d(F_{n'{}'})-\rho_d(F_{n'{}'})^2\rho_d(F_0)
\]
in the second equality below:
\[
(\rho_d(F_0))(m\otimes v)=\frac{1}{q+q^{-1}}\rho_d(F_0)\rho_d(F_{n'{}'})^2(m\otimes v')=A+B,
\]
\[
A=\rho_d(F_{n'{}'})\rho_d(F_0)\rho_d(F_{n'{}'})(m\otimes v'),\qquad B=-\frac{1}{q+q^{-1}}\rho_d(F_{n'{}'})^2\rho_d(F_0)(m\otimes v').
\]

To find $B$, we write by induction 
\[
(\rho_d(F_0))(m\otimes v')=\sum_{1\le k\le s-2}\alpha_{j'_k,F}(m)\otimes\rho^{\otimes d}(Y_{j'_kF}^{(d)})v',\qquad Y_{jF}^{(d)}=\sigma^{\otimes (j-1)}\otimes E_{\prod}\otimes K_{\prod}^{\otimes (d-j)},
\]
as $\ove_{n'}$ occurs only at the $s-2<s$ positions $j'_1,\dots,j'_{s-2}$ in $v'$. Recall that $\rho(E_{\prod})=E_{1,n'}$. Note that $\rho_d(F_{n'{}'})$ changes the factors ($\ove_{n'{}'}$ to $\ove_{n'}$) of $v'$ only at the positions $j'_{s-1}$, $j'_s$. Applying $\rho_d(F_{n'{}'})$ to $(\rho_d(F_0))(m\otimes v')$ would send the part $\ove_{n'{}'}\otimes\ove_{n'{}'}$ at the positions $j'_{s-1}$ and $j'_s$ to $\ove_{n'}\otimes q\ove_{n'{}'}$ (from the summand of $\rho_d(F_{n'{}'})$ with $(j'_{s-1},j'_s)$-parts $\rho(F_{n'{}'})\otimes \rho(K_{n'{}'}^{-1})$), plus $\ove_{n'{}'}\otimes \ove_{n'}$ (from the summand of $\rho_d(F_{n'{}'})$ with -parts $1\otimes\rho(F_{n'{}'})$). Applying $\rho_d(F_{n'{}'})$ again we obtain 
\[
\ove_{n'}\otimes q\ove_{n'}+\ove_{n'}\otimes q^{-1}\ove_{n'}=(q+q^{-1})\ove_{n'}\otimes \ove_{n'}.
\]
Now $\rho^{\otimes d}(Y_{j'_k,F}^{(d)})$ acts on the two factors $\ove_{n'}\otimes\ove_{n'}$ of $v$ at the positions $(j'_{s-1},j'_s)$ via $\rho(K_{\prod})=\diag(q,I,q)$, namely by multiplication by $q$, but not on $v'$. So in summary, 
\[
B=-q^{-2}\sum_{1\le k\le s-2}\alpha_{j'_k}(m)\otimes\rho^{\otimes d}(Y_{j'_k,F}^{(d)})v.
\]

To compute $A$, let $v'{}'$ (resp. $v'{}'{}'$) be obtained from $v'$ on replacing the vector $\ove_{n'{}'}$ at the $j'_{s-1}$ (resp. $j'_s$) position by $\ove_{n'}$.  Observe that 
\[
(\rho_d(F_{n'{}'}))(m\otimes v')=q m\otimes v'{}'+m\otimes v'{}'{}'.
\]
(Applying $\rho_d(F_{n'{}'})$ again we recover the result of the start of the proof: $(\rho_d(F_{n'{}'})^2)(m\otimes v')=(q+q^{-1})(m\otimes v)$.)
As $s(v'{}')=s-1=s(v'{}'{}')<s$, by induction we get
\[
\rho_d(F_0)\rho_d(F_{n'{}'})(m\otimes v')=q\sum_{k\not=s}\alpha_{j'_k,F}(m)\otimes\rho^{\otimes d}(Y_{j'_k,F}^{(d)})v'{}'+\sum_{k\not=s-1}\alpha_{j'_k,F}(m)\otimes\rho^{\otimes d}(Y_{j'_k,F}^{(d)})v'{}'{}'.
\]
Now we apply $\rho_d(F_{n'{}'})$. As $v'{}'$ has $\ove_{n'{}'}$ only at the $j'_s$-position, we get
\[
\rho_d(F_{n'{}'})\sum_{k\not=s}\alpha_{j'_k,F}(m)\otimes\rho^{\otimes d}(Y_{j'_k,F}^{(d)})v'{}'=q^{-1}\sum_{k\le s-1}\alpha_{j'_k,F}(m)\otimes\rho^{\otimes d}(Y_{j'_k,F}^{(d)})v.
\]
Denote this by $A_1$. As $v'{}'{}'$ has $\ove_{n'{}'}$ only at the $j'_{s-1}$ position,
\[
\rho_d(F_{n'{}'})\sum_{k\not=s-1}\alpha_{j'_k,F}(m)\otimes\rho^{\otimes d}(Y_{j'_k,F}^{(d)})v'{}'{}' =A_2+q^{-1}A_3,
\]
\[
A_2=\alpha_{j'_s,F}(m)\otimes\rho^{\otimes d}(Y_{j'_s,F}^{(d)})v,\qquad
A_3=q^{-1}\sum_{k\le s-2}\alpha_{j'_k,F}(m)\otimes\rho^{\otimes d}(Y_{j'_k,F}^{(d)})v.
\]
No factor $q^{-1}$ appears in front of $A_2$ since $\rho(K_{\prod})$ acts at positions $>j'_s$, which did not change from $v'{}'{}'$ to $v$ in $A_2$. Then $A=qA_1+A_2+q^{-1}A_3=q^{-1}A_3+A_2+qA_1$. So $B+A$ is
\[
(\rho_d(F_0))(m\otimes v)=
-q^{-2}\sum_{1\le k\le s-2}\alpha_{j'_k,F}(m)\otimes\rho^{\otimes d}(Y_{j'_k,F}^{(d)})v 
+q^{-1}\cdot q^{-1}\sum_{k\le s-2}\alpha_{j'_k,F}(m)\otimes\rho^{\otimes d}(Y_{j'_k,F}^{(d)})v
\]
\[
+\alpha_{j'_s,F}(m)\otimes\rho^{\otimes d}(Y_{j'_s,F}^{(d)})v
+q\cdot q^{-1}\sum_{1\le k\le s-1}\alpha_{j'_k,F}(m)\otimes\rho^{\otimes d}(Y_{j'_k,F}^{(d)})v
=\sum_{1\le k\le s}\alpha_{j'_k,F}(m)\otimes\rho^{\otimes d}(Y_{j'_k,F}^{(d)})v.
\]
\end{proof}
\begin{lem}\label{L16.4} 
Setting $my_j^{-1}=\alpha_{jE}(m)$, $my_j=\alpha_{jF}(m)$ defines a right $H_d^a(q^2)$-module structure on $M$, extending its $H_d(q^2)$-module structure.
\end{lem}
\begin{proof}
We have to check the following relations:\\
$(i)$ $y_jy_j^{-1}=1=y_j^{-1}y_j$; $(ii)$ $y_jy_k=y_ky_j$; $(iii)$ $y_{j+1}=\wh T_jy_j\wh T_j$.

To prove $(i)$ and $(ii)$, we compute both sides of the equality
\[
(\rho_d([E_0,F_0]))(m\otimes v)=\rho_d\left(\left[\frac{K_{\alpha_0}-K_{\alpha_0}^{-1}}{q-q^{-1}}\right]\right)(m\otimes v).
\]
For $(i)$ we take $v$ with $\ove_{n'}$ in the $j$th position and $\ove_{n'-(d-1)},\dots,\ove_{n'-1}$ in the remaining positions, in any order.

For $(ii)$ take $v$ to be a tensor with $\ove_1$ in the $j$th place, $\ove_{n'}$ in the $k$th position, and distinct vectors from $\{\ove_2,\dots,\ove_{n'{}'}\}$ in the other positions. Note that since the central element $c=K_{\alpha_0}K_{\alpha_1}\dots K_{\alpha_{n'{}'}}$ acts as 1 on every $U_q^\sigma(\EE,\Pi,p,\Gamma)$-module $W$, we have $(\rho_d(K_{\alpha_0}))(m\otimes v)=m\otimes\rho(K_{\prod}^{-1})^{\otimes d}v$.

For $(iii)$, take $v=\ove_{i_1}\otimes\dots\otimes\ove_{i_d}\in V^{\otimes d}$ with $i_j=2$, $i_{j+1}=1$, and the remaining $i_k$ are distinct from $\{3,\dots,n'{}'\}$. This is possible since $d\le n'{}'$. So: $v$ has $\ove_2$ at position $j$, $\ove_1$ at position $j+1$. The vector $v'$ is obtained from $v$ on replacing $\ove_1$ at position $j+1$ by $\ove_{n'}$. The vector $v'{}'$ is obtained from $v'$ on replacing $\ove_2$ at position $j$ by $\ove_{n'}$ and $\ove_{n'}$ at position $j+1$ by $\ove_2$. The vector $v'{}'{}'$ is obtained from $v$ on replacing $\ove_2$ at position $j$ by $\ove_1$ and $\ove_1$ at position $j+1$ by $\ove_2$. 

Now looking at the indices $(i,j)=(2,n')$ only, we have $\check R(\ove_{n'}\otimes\ove_2) =\ove_2\otimes\ove_{n'}$, and $\check R(\ove_2\otimes\ove_1)=\ove_1\otimes\ove_2$. Then
\[
m\cdot \wh T_jy_j\wh T_j\otimes v=m\cdot \wh T_jy_j\otimes v'{}'{}'=(\rho_d(F_0))(m\cdot \wh T_j \otimes v'{}')
\]
\[
=(\rho_d(F_0))(m\otimes \wh T_jv'{}')=(\rho_d(F_0))(m\otimes v')=my_{j+1}\otimes v.
\]
Since $v$ has distinct components, Lemma \ref{L16.1} implies that $m\cdot y_{j+1}=m\cdot \wh T_jy_j\wh T_j$ for all $m\in M$.

This completes the proof that $W\simeq\FF(M)$ as a $U_{q,\AI}^\sigma(\EE,\Pi,p,\Gamma)$-module.
\end{proof}

To show that $\FF$ is an equivalence we still need to show that it is bijective on sets of morphisms. Injectivity of $\FF$ follows from that of $J$. For surjectivity, let $F:\FF(M)\to\FF(M')$ be a homomorphism of $U_{q,\AI}^\sigma(\EE,\Pi,p,\Gamma)$-modules. By Lemma \ref{L16.1}, $F=J(f)$ for some homomorphism $f:M\to M'$ of $H_d(q^2)$-modules. Since $F$ commutes with the action of $F_0$ we have $(\rho(F_0)F)(m\otimes v)=(F\rho(F_0))(m\otimes v)$, i.e.,
\[
\sum_{1\le j\le d}f(m)\cdot y_j\otimes\rho^{\otimes d}(Y_{jF}^{(d)})v=\sum_{1\le j\le d}f(my_j)\otimes\rho^{\otimes d}(Y_{jF}^{(d)})v
\]
for all $m\in M$ and $v\in V^{\otimes d}$. Choosing $v$ suitably we deduce that $f(my_j)=f(m)y_j$ for all $j$ $(1\le j\le d)$. This completes the proof of theorem \ref{T11.2}.
\hfill$\square$

\section{Basic properties of $\FF$}\label{17}
Our $\FF$ is a functor of $\C$-linear categories. It commutes with induction. Write $U_{q,a}^\sigma(\sl(m,n))$ for $U_{q,\AI}^\sigma(\EE,\Pi,p)$ for simplicity.

\begin{prop}\label{P17.1}
Let $M_i$ be a finite dimensional $H_{d_i}^a(q^2)$-module $(i=1,\,2)$. Then there is a natural isomorphism $\FF(I^a(M_1,M_2))\simeq\FF(M_1)\otimes\FF(M_2)$ of $U_{q,a}^\sigma(\sl(m,n))$-modules.
\end{prop}
\begin{proof}
Let $\phi:B\to A$ be a homomorphism of associative algebras with a unit over a field, $M$ a right $B$-module, $W$ a left $A$-module, and $W|B$ is $W$ regarded as a left $B$-module via $\phi$. Then there is a natural isomorphism of vector spaces: $\ind_B^A(M)\otimes W\simeq M\otimes_BW|B$. This form of Frobenius reciprocity is given by $(m\otimes a)\otimes w\mapsto m\otimes aw$ ($m\in M$, $a\in A$, $w\in W$).

Take $A=H_{d_1+d_2}(q^2)$, $B=H_{d_1}(q^2)\otimes H_{d_2}(q^2)$, $\phi=\phi(d_1,d_2)$, $M=M_1\otimes M_2$, $W=V^{\otimes (d_1+d_2)}$, $V=V_{\0}\oplus V_{\1}$ (of dimension $n'=n+m$) being the natural representation of $U_q^\sigma(\sl(m,n))$. Note that $W\simeq (V^{\otimes d_1})\otimes (V^{\otimes d_2})$ as an $H_{d_1}(q^2)\otimes H_{d_2}(q^2)$-module. We get a natural isomorphism of vector spaces
\[
\FF(I^a(M_1,M_2))\to(M_1\otimes M_2)\otimes_{H_{d_1}(q^2)\otimes H_{d_2}(q^2)}(V^{\otimes d_1}\otimes V^{\otimes d_2}).
\]
The right side is isomorphic to $\FF(M_1)\otimes\FF(M_2)$ as a vector space. It remains to check that the resulting isomorphism $\FF(I^a(M_1,M_2))\to\FF(M_1)\otimes\FF(M_2)$ of vector spaces commutes with the action of $U_{q,a}^\sigma(\sl(m,n))$.
\end{proof}
Using the equivalence $\FF$ one can relate the universal $H_d^a(q^2)$-modules $M_c$ and for $c\in\C^\times$ the $U_{q,a}^\sigma(\sl(m,n))$-modules $V(c)$, where $V(c)$ is $V$ as a $U_q^\sigma(\sl(m,n))$-module, and $K_0=K_{\alpha_0}$ acts as $K_{\prod}^{-1}$ and $E_0$ as $c\rho(F_{\prod})=cE_{n',1}$, $F_0$ as $c^{-1}\rho(E_{\prod})=c^{-1}E_{1,n'}$.
\begin{prop}\label{P17.2}
Let $c=(c_1,\dots,c_d)\in\C^{\times d}$, $d\ge 1$, $m$, $n\ge 2$. Then there exists a natural isomorphism $\FF(M_c)\simeq V(c_1)\otimes\dots\otimes V(c_d)$. 
\end{prop}
\begin{proof}
As an $H_d(q^2)$-module, $M_c$ is the right regular representation. Hence the map $V^{\otimes d}\to J(M_c)$, $v\mapsto 1\otimes v$, is an isomorphism of $U_q^\sigma(\sl(m,n))$-modules.
\[
(\rho_d(E_0))(1\otimes v)=\sum_{1\le j\le d}1\cdot y_j^{-1}\otimes\rho^{\otimes d}(Y_{j,E}^{(d)})v=1\otimes\left(\sum_{1\le j\le d}c_j^{-1}\otimes\rho^{\otimes d}(Y_{j,E}^{(d)})\right)v.
\]
Also $\rho_d(E_0)=\sum_{1\le j\le d}(\sigma K_0)^{\otimes (j-1)}\otimes\rho(E_0)\otimes 1^{\otimes (d-j)}$ acts on $V(c_1)\otimes\dots\otimes V(c_d)$ as
\[
\sum_{1\le j\le d}(\sigma K_0)^{\otimes (j-1)}\otimes c_j\rho(F_{\prod})\otimes 1^{\otimes (d-j)}=\sum c_j^{-1}Y_{j,F}^{(d)}.
\]
The map $V^{\otimes d}\to J(M_c)$ commutes with the action of $\rho(F_0)$, $\rho(E_0)$.
\end{proof}
\begin{cor}\label{C17.3}
Let $1\le d<n'$. $(a)$ Every finite dimensional $U_{q,a}^\sigma(\sl(m,n))$-module which appears as a quotient of $V^{\otimes d}$ as a $U_{q,a}^\sigma(\sl(m,n))$-module is isomorphic to a quotient of $V(c_1)\otimes\dots\otimes V(c_d)$ for some $c_1,\dots,c_d\in\C$. $(b)$ Let $c_1,\dots,c_d\in\C$. Then $V(c_1)\otimes\dots\otimes V(c_d)$ is reducible as a $U_{q,a}^\sigma(\sl(m,n))$-module if and only if $c_j=q^2c_k$ for some $j$, $k$.
\end{cor}
\begin{proof}
This follows from the corresponding result -- Proposition \ref{P8.3} -- for the affine Hecke algebra in section \ref{8} and the fact that $\FF$ is an equivalence of categories.
\end{proof}

The Zelevinsky classification parametrizes all irreducible representations  of $\GL(n,F)$, $F$ being a $p$-adic field, in particular the $H_d^a(q^2)$-modules as the special case of the representations whose irreducible constituents have each a nonzero Iwahori-fixed vector. The equivalence $\FF$ carries this description to the category of $U_{q,a}^\sigma(\sl(m,n))$-modules.

{}

\end{document}